\documentclass{article}

\usepackage[T1]{fontenc}
\usepackage[utf8]{inputenc}
\usepackage{etoolbox}
	\csundef{endequation*}
\usepackage{authblk}

\usepackage{amssymb}
\usepackage{amsthm}
\usepackage{amsmath}
\usepackage{paralist}

\usepackage{graphics} 
\usepackage{epsfig} 
\usepackage{graphicx}
\usepackage{epstopdf}

\usepackage[colorlinks=true]{hyperref}
	\hypersetup{urlcolor=blue, citecolor=red}

\theoremstyle{definition}

\newtheorem{cor}{Corollary}
\newtheorem{thm}{Theorem}
\newtheorem{prop}{Proposition}
\newtheorem{lem}{Lemma}

\newtheorem{Remark}{Remark}
\newtheorem{defi}{Definition}

\begin{document}

\title{The morphology of MSS-sequences in a    wide class of  unimodal maps, its structure and decomposition}
    
\author[1]{Jesús San Martín}
\author[2]{Antonia González Gómez}
\author[2]{Fernando Blasco}

	\affil[1]{Universidad Politécnica de Madrid. \\ ETSIDI. Ronda de Valencia, 3. 28012 Madrid, Spain. e-mail:jsm@dfmf.uned.es. Corresponding author}
	\affil[2]{Universidad Politécnica de Madrid. \\ ETSIMFMN Avda. de Las Moreras s/n. 28040 Madrid, Spain. e-mail:\{antonia.gonzalez,fernando.blasco\}@upm.es}

\date{}

\maketitle

\begin{abstract}
The MSS-sequences (U-sequences) in a wide class of unimodal maps have the look $\mathrm{P}=(\mathrm{R} \mathrm{L}^{q})^{n_1} \mathrm{S}_1(m_1,q-1)  (\mathrm{R}   \mathrm{L}^{q})^{n_2}\mathrm{S}_2(m_2,q-1) $ $\ldots$ $ (\mathrm{R} \mathrm{L}^{q})^{n_r}$ $ \mathrm{S}_r(m_r,q-1)\mathrm{C},$ where  $\mathrm{S}_i(m_i, q-1)$ are sequences of $\mathrm{R}$s and $\mathrm{L}$s that
 contain at most  $q-1$ consecutive $\mathrm{L}$s.    The first block $\mathrm{RL}^q$ and the sequence $\mathrm{S}_1$  following it are essential for  an admissible sequence to be a MSS-sequence.
 {Moreover $\mathrm{S}_i(m_i,q-1), \ i=2, \ldots, r$  are determined  by $\mathrm{S}_1(m_1,q-1)$}.  Explicit structure of  MSS-sequences will be given  as well as the  theorems that decompose the non-primary MSS-sequences. The cardinality will be calculated for some important sets of non-primary MSS-sequences and an algorithm to generate the blocks   $\mathrm{S}_i(m_i,q-1), \ i=1, \ldots, r$
 will be provided, as the construction of the blocks $\mathrm{S}_i(m_i,q-1)$ allows the construction of the MSS-sequences.
\end{abstract}

\section{Introduction}
The key point of this paper is focused on making explicit the structure and  the construction of MSS-sequences of one-dimensional discrete systems 
\begin{equation} \label{eqn1}
x_{n+1}=f_\lambda(x_n),  \ f_\lambda: \mathrm{I}\rightarrow \mathrm{I},  \ \mathrm{I}=[a,b]
\end{equation}
ruled by unimodal functions \cite{collet} {(
the conditions that those unimodal maps should satisfy will be stated later). This goal is  motivated both by physical and mathematical reasons.

Many physical systems are strongly dissipative because their flows are more contracted along the stable manifolds than are expanded along the unstable manifolds around the equilibria, as a result flows can be characterized through one-dimensional return maps. Furthermore, if the contraction rate is strong enough, then one can consider that the return map turns out to be unimodal for all practical purposes, even for high-dimensional flows \cite{cv}. The advantage of  physical dynamical systems ruled by  unimodal map is that these systems show an universal behavior under rather general conditions: all maps have the same bifurcation diagram \cite{ref_1}, always appear the same sequences {(MSS-sequences)} with the same order of occurrence \cite{collet, ref_2, beyer, wang}
(an algorithm for  generation of  these sequences is given in reference \cite{ref_2}),  the combinatorial properties of the system determine the geometrical properties discovered by Feigenbaum \cite{ref_3}. Therefore, the physical systems inherit this universal behavior, resulting that highly dissipative dynamical systems, ruled by very different differential equations, can be addressed as a single one.
\par
The shape of the unimodal map induces a natural partition  between left and right of its critical point (denoted by $\mathrm{C}$),  partitions are labeled as  $\mathrm{L}$ (left) and $\mathrm{R}$ (right). As results,  the iterates of a point by $f_{\lambda}$ are coded by a sequence of symbols $\mathrm{R}$  and $\mathrm{L}$: { the  itinerary of the point \cite{collet}.
The opposite is not true since not every sequence
$\mathrm{R}$s and  $\mathrm{L}$s (admissible sequence \cite{collet}) is associated to the itinerary of a point. This  combinatorial description of the dynamics, with $\mathrm{R}$s and  $\mathrm{L}$s,   goes back to the work of {Beyer, Mauldin and  Stein -BMS-}, who gave  a very simple criterion \cite{beyer} for recognizing whether or not an {admissible sequence is a MSS-sequence: for the class of unimodal round-top, concave functions, if the admissible sequence is shift-maximal then it is a MSS-sequence.}
\par
Nonetheless, despite the fact that much time has passed since the criterion was established, and the very simplicity  of the  criterion,  there is one question that still remains open and that it is necessary to answer in order to complete the combinatorial description. This question is, what are specifically { these universal MSS-sequences}? And, in particular, how are they built and what is the relevant information derived from them?
We shall answer these questions in order to complete the combinatorial description of these systems.
\par
From a mathematical point of view, the importance of solving this problem is not only  found in completing the combinatorial description, but the dynamical system is  completely characterized by these sequences, so  it is necessary to understand them.
\par
Building patterns by  the traditional trial and error method might be assumed at first glance the natural approach to finding the MSS-sequences. As the grammar of this kind of sequences has only two letters,  $\mathrm{R}$ (right) and  $\mathrm{L}$ (left), one would be tempted to combine the two letters in order to build  sequences and then using the BMS-criterion for recognizing whether or not an admissible sequence is a MSS-sequence. This approach would be hopeless. The number of patterns grow exponentially (variations with repetition)  with the period of sequences and the problem becomes rapidly intractable. Sequences of length  as short as $15$ generate $2^{15}$ different patterns, and it is totally useless to look for patterns from which we can derive some specific rule.
\par
We need a rule that dramatically decreases the number of the posible patterns in which to focus our attention. Let us note that the itinerary of the critical point (kneading sequence) belongs to $[f^2(\mathrm{C}), f(\mathrm{C})]$ and in particular the minimum value  $f^2(\mathrm{C})$   corresponds to the first $\mathrm{L}$ of the sequence, hence, any other  $\mathrm{L}$ of the sequence will correspond to a bigger value. On the other hand, $\mathrm{L}$s  of the sequence correspond to points belonging to  $[f^2(\mathrm{C}),\mathrm{C})$  where $f$ es increasing. So, if $f^2(\mathrm{C})$  leaves the interval  $[f^2(\mathrm{C}),\mathrm{C})$  after $q$ iterations, then any other  sequence point belonging to $[f^2(\mathrm{C}),\mathrm{C})$  ---that is, points associated with $\mathrm{L}$s--- will leave the interval after $q$ iterations at most. It follows the well-known result that {MSS}-sequences cannot have consecutive sequences of $\mathrm{L}$s  longer than the first consecutive sequence of $\mathrm{L}$s. As all sequences start as $\mathrm{CRL}^q$, its consecutive sequence of $\mathrm{L}$s  will have $q$ $\mathrm{L}$s at the most. That will determine the first brick to complete the building: the  $\mathrm{RL}^q$ block. Obviously, between two consecutive  $\mathrm{RL}^q$ blocks there cannot be consecutive sequences of $\mathrm{L}$s longer that $q-1$. Therefore, the structure of the {MSS}-sequences is 
$ (\mathrm{RL}^q)^{n_1} \mathrm{S}_1 (\mathrm{RL}^q)^{n_2} \mathrm{S}_2 \ldots (\mathrm{RL}^q)^{n_k} \mathrm{S}_k \mathrm{C}$. Where  $\mathrm{S}_i,  i=1,..,k$  are sequences of $\mathrm{R}$s  and $\mathrm{L}$s, with at most $q-1$ consecutive $\mathrm{L}$s. 
The original problem has now become: 1.- determine the values of $n_i$ that give the possibility of having MSS-sequences,  2.- determine the sequences $\mathrm{S}_i$.
We find that, surprisingly, the blocks $\mathrm{S_i}$, $i>1$ are controlled by $\mathrm{S}_1$, which is located between the first two $\mathrm{RL^q}$ blocks (both $\mathrm{S_1}$ and $\mathrm{S_i}$ are calculated in section IV). Broadly speaking, MSS-sequences are built by linking sequences $(\mathrm{RL}^q)^{n_i} \mathrm{S}_i$,  with $\mathrm{S}_i$  ruled by $\mathrm{S_1}.$ What seemed to be an intractable puzzle of combinations will be reduced to the combination of two blocks according to a far more restrictive rules than the original problem, it will allow us to obtain our goal of determining the explicit structure of the MSS-sequences  (section III).
\par
Obviously, once MSS-sequences have been identified, the following step will be to study how these structures are composed in the sense of Derrida, { Gervois and Pomeau} -DGP- \cite{De}, that is, we will identify the non-primary sequences and by using several theorems will decompose them as compositions of primary sequences (section V). Notice that characterizing  primary periodic sequences is characterizing the basic bricks with which the bifurcation diagram is built, because all periodic sequences of the diagram are either primary or the composition of primary sequences.
\par
Finally, we will be in a position to calculate the cardinality of some sets
of non-primary sequences (section VI). 
\par
In section VII we will indicate how those results can be useful to solve open problems in dynamical systems.
\par
\section{\label{sec:defs}Definitions, notations and previous theorems.}

 {Let   $\mathrm{P}=\mathrm{A_1}\, \mathrm{A}_{2} \,  \ldots  \,\mathrm{A}_{i}\, \ldots \mathrm{A}_p$  be a finite sequence where $\mathrm{A}_{i}=\mathrm{R}$ or $\mathrm{L}$ for $i=1, \ldots, p-1$ and  $\mathrm{A}_p=\mathrm{C}$. Those sequences are called admissible   \cite{beyer}. BMS 
 defined  a  linear  order  on  $\mathrm{P}$ {according to Collet and Eckman  \cite{collet}. They call  this  linear  order  a  parity-lexicographic    ordering.     First,  put  $\mathrm{L}  < \mathrm{C} < \mathrm{R}$.  Let   $\mathrm{P^{1}}$ and $\mathrm{P^{2}}$ be two sequences in   $\mathrm{P}$. Let $i$  be the first  index  where they differ, $\mathrm{A^{1}}_{i}\neq \mathrm{A^2}_{i}$. If 
$i=1$     then   $\mathrm{P^{1}}< \mathrm{P^{2}}$    iff  $\mathrm{A^{1}_{1}}<  \mathrm{A^{2}_{1}}$.     Suppose  $i>1$.    In   case  $\mathrm{A^1_1}\, \mathrm{A^1}_{2} \,  \ldots  \,\mathrm{A^1}_{i-1}= \mathrm{A^2_1}\, \mathrm{A^2}_{2} \,  \ldots  \,\mathrm{A^2}_{i-1}$,  have an even number of $\mathrm{R}$s then $\mathrm{P^{1}}< \mathrm{P^{2}}$ iff $\mathrm{A^{1}_i}< \mathrm{A^{2}_i}$  and in case there are an odd  number  of  $\mathrm{R}$s,  then  $\mathrm{P^{1}}< \mathrm{P^{2}}$  iff  $\mathrm{A^{2}_i}< \mathrm{A^{1}_i}$. 
An  admissible sequence $\mathrm{P}$ is called  shift  maximal  if  it  is greater than  or 
equal to  each of its  right  shifts. }

The iterates of a point are easily associated with admissible sequence by using the itinerary of the point. Given $f:[0,1] \to [0,1]$, the itinerary \cite{collet,beyer}  of the point  $x \in [0,1]$ is the admissible finite sequence $\mathrm{I^f} (x)=\mathrm{P}=
\mathrm{A_1}\, \mathrm{A}_{2} \,  \ldots  \,\mathrm{A}_{i}\, \ldots \mathrm{A}_p$,
where $\mathrm{A_i}=\mathrm{R (L)}$ if $f^i(x)>\frac{1}{2} \ (<\frac{1}{2}) $ and $\mathrm{A_i} =\mathrm{C}$ if $f^i (x) = \frac{1}{2}.$

{An admissible sequence is turned into a sequence of numbers by using the $\lambda$-sequence (the $\lambda$-sequence eases the comparisons and make more compact proofs).

\begin{defi}\label{d1}   \cite{De}
{Let  $\mathrm{P}=\mathrm{A_1}\, \mathrm{A}_{2} \,  \ldots  \,\mathrm{A}_{i}\, \ldots \mathrm{A}_p$  be an admissible  sequence.} 
Let $\mathrm{\beta}(\mathrm{A_i})$ be  the number of $\mathrm{R}$s previous to 
$\mathrm{A}_{i}. $
The   $\lambda-$sequence of $\mathrm{P}$, denoted by $\lambda_{\mathrm{P}}$  or  $\lambda_{\mathrm{A}_1,\ldots, \mathrm{A}_p},$ is the sequence
$(a_1, \ldots,a_{p-1},a_p)$ with 
$$
a_i =
\begin{cases}
(-1)^{\beta(A_{i})} & \text{ if }  \mathrm{A}_i=\mathrm{R} \\
(-1)^{\beta(A_i) + 1} & \text{  if }  \mathrm{A}_i=\mathrm{L} \\
0 &\text{ if } \mathrm{A}_i=\mathrm{C} \\
\end{cases}
$$
\end{defi}

Given  $\lambda_\mathrm{P}=(a_1, a_2, ....a_{p})$, the  \emph{shift operator} $\sigma$  is defined  
as usual by
$\sigma^k(\lambda_{\mathrm{P}})=\sigma^k(a_1, a_2, \ldots , a_{p})=(a_{k+1}, a_{k+2}, \ldots , a_{p}, \underbrace{0,\ldots,0}_{k})$ for $k=1, \ldots,p$.  
 Given the sequence $\mathrm{P}=\mathrm{A}_1\, \mathrm{A}_{2}  \ldots\mathrm{A}_p$  it follows that  $\sigma^k({\mathrm{P}})=\sigma^k(\mathrm{A}_1  \ldots  \mathrm{A}_{p})=\mathrm{A}_{k+1}\, \ldots \, \mathrm{A}_{p}$ 

\medskip

We have that $(a'_1, a'_2, \ldots , a'_{p-1})<(a_1, a_2, \ldots , a_{p-1})$
 if  $a'_i< a_i$, where $i$ is the least integer $i$ such that  $a'_i\neq a_i$.
 An useful method when comparing two $\lambda-$sequences is identifying the place where they begin to be different. Graphically 
we will write both $\lambda-$sequences  in parallel  with a vertical line 
 indicating the place where they start to be different
$$
\begin{array}{r|c}
(a_1, a_2, \ldots , a_{i-1}, & a'_{i} \ldots , a'_{p-1})\\(a_1, a_2, \ldots , a_{i-1},& a_{i}, \ldots ,a_{p-1})\end{array}
$$

Every MSS-sequence starts with
 $\mathrm{RL}^q$, 
 therefore from now on we will focus on sequences $\mathrm{RL}^q\mathrm{C}$ if their length is $p=q+2$ and 
$ \mathrm{P}=\mathrm{R}\, \mathrm{L}^{q} \,  \mathrm{R} \,\mathrm{A}_{q+3}\, \mathrm{A}_{q+4}\cdots \mathrm{A}_{p-1}\mathrm{C},$
otherwise. The last one has as   $\lambda-$sequence
  $\lambda_\mathrm{P}=({\underbrace{1}_{\substack{\mathrm{R}}}},
{\underbrace{1,\ldots,1}_{\substack{\mathrm{L}^{q}}}}, \underbrace{-1}_{\substack{\mathrm{R}}},\ldots)$

Notation. \label{re1} For convenience, we will use the following notations:
 
 a) $1_k$ $(-1_k)$ will denote a consecutive sequence with $k$ $1$s ( $-1$s). 
 
 b) $\pm \, 1_k$ ($\mp \, 1_k$) an alternated sequence consisting on $1$s and $-1$s,  starting with $+1$ ($-1$) and  length $k$ .

 c)  $0_k$ will denote a consecutive sequence with $k$ $0$s.
 
 With the new notation, 
  $$\lambda_P=({\overbrace{1}^{\substack{\mathrm{R}}}},
{\overbrace{1,\ldots,1}^{\substack{\mathrm{L}^{q}}}}, \overbrace{-1}^{\substack{\mathrm{R}}},\ldots)=(1_{q+1},-1, \dots)$$

\begin{defi}  \cite{beyer}
An  unimodal  round  top concave map is
an unimodal and  continuous map
{ $F:\mathrm{[0,1]}\to\mathrm{[0,1]}$   such that}

a) $F (0)=F (1)=0$, $F(\frac{1}{2}) = 1$,
 $F$ is  nondecreasing on $[0,\frac{1}{2}]$ and non-increasing on $[\frac{1}{2}, 1]$

b) is concave 

c)  there  exists 
$ e\in (0,\frac{1}{2})$  such that  $F '$   exists and is continuous  in  $(e, 1-e)$   and $F'(\mathrm{\frac{1}{2}})   = 0$.

\end{defi}

\begin{thm}{\label{t1}}  \cite{beyer}
 { Let $F$ be a unimodal round-top  function.  For  each shift-maximal  sequence  $\mathrm{P}$ there is a value of $\lambda$ such 
that    $\mathrm{I}^{\lambda \, F} (\lambda) = \mathrm{P}$. In  particular,  each MSS sequence occurs. }
\end{thm}

\begin{thm}{\label{t1s}} \cite{beyer}   {Let   $F$   be  unimodal.    For   any   $\lambda \in (0,1)$,  $\mathrm{I}^{\lambda F}(\lambda)$    is  shift maximal.   In  particular,    an  MSS-sequence is  shift    maximal. }
\end{thm}

In the construction of the MSS-sequences we well find that some patterns are not shift maximal and, according to theorem \ref{t1s}, they are not MSS-sequences, thus we will reject those patterns, whereas theorem \ref{t1} will be used later to obtain the explicit aspect of MSS sequences. From now on we will work with the unimodal maps that verify the conditions given in theorem \ref{t1}.

{Now the question is dealing with the shift-maximal sequences that appear in theorem \ref{t1}. We need an operational method that allows to make explicit  the structure of the MSS-sequences. This operational method was given by DGP
\cite{De}, that translated the admissible sequences $\mathrm{P}$ into number sequences, the so-called $\lambda-$sequences $\lambda _\mathrm{P}$; using the sequence $\lambda _{\mathrm{P}}$, the shift-maximality condition $\sigma^k(\mathrm{P})<\mathrm{P}$ for each $k$ is expressed as $\pm \sigma^k(\lambda_{\mathrm{P}})< \lambda _{\mathrm{P}}$ for each $k$ (DGP \cite{De} give a theorem that allows constructing the MSS-sequences in the way it is done in this paper, but the conditions given in theorem \ref{t1}  are weaker, so we use the formulation of BMS instead of DGP).
The $\lambda_{\mathrm{P}}$ sequence has an operational advantage to the admissible sequence $\mathrm{P}$ when we make comparisons in order to check if $\mathrm{P}$ is shift-maximal. When we compare $\sigma^k(\mathrm{P})$ and $\mathrm{P}$ we must calculate the parity of the fragment that is common to $\sigma^k(\mathrm{P})$ and $\mathrm{P}$ and, obviously, this must be done for each $k.$ When we use $\lambda _{\mathrm{P}}$ we do not calculate the parity of the common fragment, we just study the worst case for $+\sigma^k(\lambda _{\mathrm{P}})$ and $-\sigma(\lambda _{\mathrm{P}})^k$ and, consequently, the proofs are simpler. 
Remark that either $+\sigma(\lambda _{\mathrm{P}})$ or $-\sigma(\lambda_ {\mathrm{P}})$ begins with $-1$ and it is always less or equal to $\lambda_{\mathrm{P}}$, so we only have to study one case: the case with the worst conditions. On the other hand, the presence of the sign $\pm$ is easy to understand: $\lambda _{\sigma^k(\mathrm{P})}$ can have the opposite sign to $\sigma^k(\lambda _{\mathrm{P}})$ since $\sigma^k(\mathrm{P})$ can change the parity of the common fragment when it is moved to the first position. In addition to the operational advantage derived from the use of $\lambda$-sequences, its use allows a more compact notation and a very simple comparison procedure, as  we have seen above.}

\section{The morphology and structure of the MSS-sequences}

\begin{Remark}\label{re12}
{Given the  {admissible sequence}
$\mathrm{P}=\mathrm{R}\, \mathrm{L}^{q} \,  \mathrm{R} \,\mathrm{A}_{q+2} 
\cdots$  $\mathrm{A}_{p-1}\mathrm{C},$
 either  $\sigma^n(\lambda_{\mathrm{P}} )$ or   $-\sigma^n(\lambda_{\mathrm{P}} )$ starts with $-1$. The one starting with $-1$ always verifies  the condition described in theorem {\ref{t1}} as  $\lambda_\mathrm{P}=(1_{q+1}, {-1},\ldots)$ {i.e. either
 $-\sigma^n(\lambda_{\mathrm{P}})<\lambda_{\mathrm{P}}$
or  $\sigma^n(\lambda_{\mathrm{P}})<\lambda_{\mathrm{P}}.$}
 So, to know whether a sequence is shift maximal or not we only need to pay attention to those shifts  $\pm \, \sigma^n(\lambda_\mathrm{P} )$ beginning with $1$. Thus, without loss of generality, we will  always assume that  $\sigma^n(\lambda_\mathrm{P})$  is the sequence that starts with $1$. 
The case $n=p$ must be treated separately since 
we have
$\pm \,\sigma^n(\lambda_{\mathrm{P}} )=(0_p)$, so  $\sigma^p(\lambda_\mathrm{P} )< \lambda_\mathrm{P}$. 
}

Notice that \ $\sigma^k(\lambda_\mathrm{P} )\ = \ (\lambda_{\sigma^k{(\mathrm{P})}}, \overbrace{0,\ldots, 0}^{p-k}),$ 
however
$\sigma^k({\mathrm{P}})=\sigma^k(\mathrm{A}_1  \ldots  \mathrm{A}_{p})=$ $\mathrm{A}_{k+1}\, \ldots \, \mathrm{A}_{p} $ without filling with $0$s in the end (see Definition \ref{d1}). The convenience of this fact will be used in proofs. Since $0$s do not play any role in the proof, by abuse of notation we will write that  $\sigma^k(\lambda_\mathrm{P} )= \lambda_{\sigma^k{(\mathrm{P})}},$ 
using them interchangeably.
\end{Remark}

\begin{lem} \label{l1}
The  {admissible  sequences} $ \mathrm{P}=\mathrm{R}\, \mathrm{L}^{q} \,  \mathrm{R} \ldots  \mathrm{A}_{h} $ $ \mathrm{A}_{h+1}\ldots \mathrm{A}_{h+j}\ldots \mathrm{A}_{p-1}\mathrm{C}$  such that $\mathrm{A}_{h}=\mathrm{R}$ and $\mathrm{A}_{h+j}=\mathrm{L}$ for all  $j=1,\ldots,s$ and $s > q$  { are  not shift maximal.}
\end{lem}

\begin{proof}
Let us write $s=q+(s-q).$ 
It results, by using definition \ref{d1}, that
$$\begin{array}{r|l}
\lambda_\mathrm{P}=(1,{ 1}_{\substack{q}}, & -1, \ldots \ldots)\\
  &\ {\mathbf{\wedge }}  \\
   \sigma^{h-1}(\lambda_\mathrm{P})=(1,{1}_{\substack{q}}, & \ {1}_{{\substack{s-q}}},   \ldots,  ) 
 \hbox{}\end{array}$$

So Theorem  \ref{t1} is not satisfied. 
\end{proof}

\medskip

The lemma states the following: if the first
$\mathrm{R}$ 
of the sequence is followed by $q$ consecutive 
$\mathrm{L}$s 
then a necessary condition for  a sequence to be MSS-sequence is 
that it has a series of, at most, $q$ consecutive 
$\mathrm{L}$s. As we have remarked in the introduction, this is a well known fact that will lead us to an important statement.

\begin{defi}{\label{dS}}    We denote by  $\mathrm{S}(m,h)$ {the set of sequences}  consisting of  $\mathrm{R}$s and $\mathrm{L}$s, with length   $m$, starting with $\mathrm{R}$ 
and containing at most 
$h$ consecutive $\mathrm{L}$s.   $\mathrm{S}(0,h)$ is the empty set. 
  We denote $\mathrm{S}(m_i,q-1)$ by   
   $\mathrm{S}_i$.
 \end{defi}

Having in mind definition \ref{dS} and lemma \ref{l1}, the candidates to MSS-sequences must follow the pattern

 \begin{equation}{\label{ec2}}
\mathrm{P}=(\mathrm{R} \mathrm{L}^{q})^{n_1}\, \mathrm{S}_1(\mathrm{R} \mathrm{L}^{q})^{n_2}\mathrm{S}_2\ldots (\mathrm{R} \mathrm{L}^{q})^{n_r} \mathrm{S}_r\mathrm{C} 
 \end{equation}
Thus, in order to get the MSS-sequences we need to know the values of  $\mathrm{n}_i$  and the $\mathrm{S}_i$.

\begin{prop}{\label{p1}}  Let  be the { admissible  sequences}  $\mathrm{P}=(\mathrm{R} \mathrm{L}^{q})^{n_1}\, \mathrm{S}_1$ $(\mathrm{R} \mathrm{L}^{q})^{n_2}\mathrm{S}_2\ldots (\mathrm{R} \mathrm{L}^{q})^{n_r} \mathrm{S}_r\mathrm{C}$. If $n_1\geq 2$ or $\mathrm{S}_r
= \mathrm{S}(0,q-1)$ with $r\geq 2$ then $\mathrm{P}$ 
  { are  not shift maximal.} 
\end{prop}

\begin{proof}
$i)$\,  If  $n_1\geq 2$  then  the block $(\mathrm{RL}^q)^{n_1}$ implies that $\lambda_\mathrm{P}$ starts with $n_1$ sequences $1_{q+1}$ with alternating sign. If we take $n= (n_1-1)(q+1)$
then
$\sigma^n(\mathrm{\lambda_\mathrm{P}})$ shifts $n_1 -1$ blocks with length $q+1$, i.e., we shift every block $\mathrm{RL}^q$  of $(\mathrm{RL}^q)^{n_1}$
except the last one, which generates a $1_{q+1}$ in the $\mathrm{\lambda}-$sequence.

On the other hand, $\mathrm{\lambda}_{\mathrm{S}_1}$ 
starts with $1_k,$
$k<q$ because $\mathrm{S}_1$ has $q-1$ consecutive $\mathrm{L}s$
at most.

Notice that $\mathrm{S}_1$ will be preceded by $\mathrm{RL}^q$ after the shift, so the sign of its $\mathrm{\lambda-}$sequence will change to $-1$.

Writing  $-1_{q+1}=-1_k -1_{q+1-k}$ it results
$$\begin{array}{r|l}
 \lambda_\mathrm{P}=({ 1}_{\substack{q+1}},  {-1}_{k} & {-1}_{q+1-k} \ldots ) \hbox{}   \\ 
  & \ \  {\mathbf{\wedge }}  \\
  \sigma^n (\lambda_P)=({1}_{\substack{q+1}}, {-1}_{{\substack{k}}} & \ \  1   \ldots,  )
  \end{array}\\
  $$ $$
 \Longrightarrow \lambda_\mathrm{P} < \sigma^n (\lambda_\mathrm{P}) \ \Longrightarrow \mathrm{P} \ \text{ is not { shift maximal}} $$
 
$ii)$
\, Let  $n$ be such that $\sigma^n(\mathrm{P})=(\mathrm{R} \mathrm{L}^{q})\, \mathrm{S}_r$. If $\mathrm{S}_r= \emptyset$
then
$$\begin{array}{r|l}
 \lambda_\mathrm{P}=({ 1}_{\substack{q+1}} & {-1}_{q+1} \ldots ) \hbox{}\\
  &\ {\mathbf{\wedge }}  \\
  \sigma^n (\lambda_\mathrm{P})=({1}_{\substack{q+1}} & \ {0}_{{\substack{m_r}}} \  0  )\\ 
\end{array}
$$ $$ \Longrightarrow \lambda_\mathrm{P} < \sigma^n (\lambda_\mathrm{P}) \ \Longrightarrow \mathrm{P} \ \text{ is not { shift maximal}} $$
\end{proof}

The Proposition \ref{p1} has reduced the candidates to MSS-sequences to the following  patterns
$\mathrm{P}=\mathrm{R} \mathrm{L}^{q}\mathrm{S}_1$ $(\mathrm{R} \mathrm{L}^{q})^{n_2}\mathrm{S}_2\ldots (\mathrm{R} \mathrm{L}^{q})^{n_r} \mathrm{S}_r \mathrm{C}$ with $\mathrm{S}_r \neq \emptyset$
and
\ $\mathrm{P}=\mathrm{R} \mathrm{L}^{q} \mathrm{C}$.
Since the latter are { shift maximal} we only have to study the sequences \begin{equation}\label{s2}
\mathrm{P}=\mathrm{R} \mathrm{L}^{q}\, \mathrm{S}_1(\mathrm{R} \mathrm{L}^{q})^{n_2}\mathrm{S}_2\ldots (\mathrm{R} \mathrm{L}^{q})^{n_r} \mathrm{S}_r \mathrm{C} 
 \end{equation} 

\begin{Remark}{\label{r3}}
From $(\ref{s2})$ it results 
 $\lambda_{\mathrm{P}}=( 1_{q+1} -1_k\ldots )$ with $k<q$ as $\mathrm{S}_1$ has $q-1$ consecutive $\mathrm{L}$s at most.
 Thus the shifts generating sequences
 starting with $1_k$ or $-1 _k$, $k<q+1,$ 
 verify the condition described in Theorem \ref{t1}, 
 consequently we only need to pay attention to sequences starting with  $1_{q+1}$.

Notice the following particular cases:
 
 i) If \ $\sigma^n(\mathrm{P})= (\mathrm{RL}^q)^k \mathrm{S}_i\cdots \mathrm{S}_r\mathrm{C}$ with  $k\geq 2$ it follows $\lambda_{\sigma^n (\mathrm{P})}=( 1_{q+1}, -1_{q+1}, \ldots)< \lambda_{\mathrm{P}}$.
 
 \smallskip
 
ii) If \ $\sigma^n(\mathrm{P})= \mathrm{\widehat{S}}_i(\mathrm{RL}^q)^{n_{i+1}}\cdots \mathrm{S}_r\mathrm{C}$ ,  where $\mathrm{\widehat{S}}_i\subset \mathrm{{S}}_i,$ as    $\mathrm{S}_i$ has, at most,  $q-1$ consecutive $\mathrm{L}$s, then $\lambda_{\sigma^n(\mathrm{P})}=( 1_{k},  \ldots)< \lambda_{\mathrm{P}}, \ k<q.$

\smallskip

iii) If \ $\sigma^n(\mathrm{P})= \mathrm{L}^j\mathrm{S}_i\cdots \mathrm{S}_r\mathrm{C}$ with $j=1,\ldots ,q$   then  $\lambda_{\sigma(\mathrm{P})}=( 1_{j},  \ldots)<( 1_{q+1},  \ldots)= \lambda_{\mathrm{P}}.$   

\smallskip

It follows that the only shifts we have to pay attention to are those given by
$\sigma^n(\mathrm{P})= \mathrm{RL}^q \mathrm{S}_i\ldots \mathrm{S}_r\mathrm{C}$ with $\lambda_{{\sigma^n({\mathrm{P}} )}}=( 1_{q+1}, -\lambda_{\mathrm{S}_i},\ldots)$ 
(the change of the sign of $\lambda_{\mathrm{S}_i}$ is due to the $\mathrm{R}$ in the block $\mathrm{RL^q}$ that precedes it) 
and deduce under which conditions those shifts verify that
 $\lambda_{\sigma^n(\mathrm{P})}<\lambda_{\mathrm{P}}$. 
 This will be done in the following theorems.
 
 Two steps are needed. First, study the structure of the  $\mathrm{S_i}$ sequences, $i\geq 2$,
 that, as we will see, are determined by the sequence  $\mathrm{S}_1\mathrm{RL}^q$.
 Second, study the restrictions on the ${n_k}$
  of \ $(\mathrm{RL}^q)^{n_k}.$
 The apparent simplicity of expression $(\ref{s2})$ is tricky.
 Although the first block $\mathrm{RL}^q$  and $\mathrm{S}_1$ 
 are the blocks that will determine if it is shift maximal, there are combinations with repeated blocks that make it necessary to study  $\lambda-$sequences 
 longer than the corresponding to the $\lambda-$sequence of $\mathrm{RL}^q\mathrm{S}_1$. 
 Let us proceed by parts: first of all we shall find the MSS-sequences without a repeated block $\mathrm{RL}^q$.
\end{Remark}

\smallskip

 \begin{thm}\label{t2} The { admissible  sequences}  $$\mathrm{P}=\mathrm{R}\, \mathrm{L}^{q} \, \mathrm{S}(m,q-1) \mathrm{C}$$  are { shift maximal.}
   \end{thm}
   
 \begin{proof}
 It is straightforward by applying theorem 1 as  $\mathrm{S}(m,q-1)$ 
   contains
   ${q-1}$ consecutive $\mathrm{L}$s  at most. 
\end{proof}

\medskip

Note that theorem \ref{t2}, according to theorem \ref{t1}, has provided the set of MSS-sequences with just one group $\mathrm{R L}^q$. So the  next step consists on finding the  {MSS-}sequences in which the group  $\mathrm{R L}^q$
appears more than once.}

\begin{thm}\label{t4}
Let    
$\mathrm{P}=\mathrm{R} \mathrm{L}^{q}\, \mathrm{S}_1 \cdots\mathrm{S}_k (\mathrm{R} \mathrm{L}^{q})^{n_{k+1}}\mathrm{S}_{k+1} \cdots $ $ (\mathrm{R} \mathrm{L}^{q})^{n_r}\mathrm{S}_r \mathrm{C}$ be   admissible  sequences and $k$,   {$1\le k<r$}, such that
$$
\mathrm{RL}^q\,\mathrm{S}_{k+1}(\mathrm{R} \mathrm{L}^{q})^{n_{k+2}}\mathrm{S}_{k+2}\ldots (\mathrm{R} \mathrm{L}^{q})^{n_{k+j}}= \mathrm{RL}^q \, \mathrm{S}_1(\mathrm{R} \mathrm{L}^{q})^{n_2}\mathrm{S}_2\ldots \mathrm{S}_{j-1}(\mathrm{R} \mathrm{L}^{q})^{n_{j}}
$$
with $\mathrm{S}_{j}\neq  \mathrm{S}_{k+j}.$ 
If 
  $(-1)^{\mathrm{\beta}(\mathrm{S}_j)}\lambda_{\mathrm{S}_j(\mathrm{R L^q})}> (-1)^{\mathrm{\beta(S_{j})}} \lambda_{\mathrm{S}_{k+j}}$ then $\lambda_{\sigma^n(\mathrm{P})}<\lambda_{\mathrm{P}}$ where $\sigma^n(\mathrm{P})= \mathrm{RL}^q \mathrm{S}_{k+1}\ldots \mathrm{S}_r\mathrm{C}$.      
  \end{thm}

\begin{proof}

a) $j=1$. This is the case  $\mathrm{S}_{k+1}\neq \mathrm{S}_{1}$ with $2 \le k+1 \le r$, i.e., the sequence   $\mathrm{RL}^q$ is the only repeated block. 
By hypothesis we have that   $-\lambda_{\mathrm{S}_1 \mathrm{RL}^q} > - \lambda_{\mathrm{S}_{k+1}}$  since $\mathrm{S}_1$ is preceded by just one $\mathrm{R}$,
 so

 $$\begin{array}{r|l} 
 \lambda_{\mathrm{P}}=({ 1}_{\substack{q+1}}, & - \lambda_{\mathrm{S_1}\mathrm{R} \mathrm{L}^{q}}\ldots \ldots) \\
  &\quad {\mathbf{\vee}}  \\
  \lambda_{\sigma^n (\mathrm{P})}= ({ 1}_{\substack{q+1}},  & -\lambda_{\mathrm{S}_{k+1}} \ldots \ldots)\\ 
\end{array}
\Longrightarrow \lambda_{\sigma^n (\mathrm{P})} <  \lambda_{\mathrm{P}}  $$

b) Let $n$ be such that $\sigma^n(\mathrm{P})= \mathrm{RL}^q \mathrm{S}_{k+1}\ldots\mathrm{S}_{k+{j-1}} (\mathrm{R} \mathrm{L}^{q})^{n_{k+j}}\ldots \mathrm{S}_r\mathrm{C}$  it 
follows that 
\begin{eqnarray}
\lambda_{\sigma^n(\mathrm{P})}& = & ({ 1}_{\substack{q+1}}, - \lambda_{\mathrm{S}_{k+1}(\mathrm{R} \mathrm{L}^{q})^{n_{k+2}}\ldots (\mathrm{R} \mathrm{L}^{q})^{n_{k+j}}}, {(-1)^{\mathrm{\beta}(\mathrm{S}_{k+j})}} \lambda_{{\mathrm{S}_{k+j}}} \ldots ) \\ 
&=&({ 1}_{\substack{q+1}}, - \lambda_{\mathrm{S}_{1}(\mathrm{R} \mathrm{L}^{q})^{n_{2}}\ldots (\mathrm{R} \mathrm{L}^{q})^{n_{j}}}, {(-1)^{\mathrm{\beta}(\mathrm{S}_{j})}} \lambda_{{\mathrm{S}_{k+j}}} \ldots ) \notag
\end{eqnarray} 

and having in mind the hypothesis in the theorem it results that

$$\begin{array}{r|l}
 \lambda_{\mathrm{P}}=({ 1}_{\substack{q+1}}, - \lambda_{\mathrm{S}_1(\mathrm{R} \mathrm{L}^{q})^{n_2}\mathrm{S}_2\ldots \mathrm{S}_{j-1}(\mathrm{R} \mathrm{L}^{q})^{n_{j}}} & {(-1)^{\mathrm{\beta}(\mathrm{S}_j)}} \lambda_{{\mathrm{S}_j(\mathrm{R} \mathrm{L}^q)}} \ldots )\\
  &\ {\mathbf{\vee}}  \\
 \lambda_{\sigma^n(\mathrm{P})}= ({ 1}_{\substack{q+1}}, - \lambda_{\mathrm{S_{k+1}}(\mathrm{R} \mathrm{L}^{q})^{n_{k+2}}\mathrm{S_{k+2}}\ldots (\mathrm{R} \mathrm{L}^{q})^{n_{k+j}}} & {(-1)^{\mathrm{\beta(S_{j})}}} \lambda_{{\mathrm{S_{k+j}}}} \ldots )\hbox{}\\ 
\end{array} \Longrightarrow \ \lambda_{\sigma^n(\mathrm{P})}<\lambda_{\mathrm{P}}$$
\end{proof}

 In this proof we are not  paying attention to the shifts 
  that do not pose any problem, such as detailed in Remark
   \ref{r3}.
  
Notice that $\lambda_{\mathrm{S}_{k+j}}$ is multiplied by  ${(-1)^{\mathrm{\beta}(\mathrm{S}_{j})}}$  since  $\lambda_{\mathrm{S}_{k+j}}$ in $\lambda_{\sigma^n(\mathrm{P})}$ 
is preceded by a sequence identical to the one preceding 
$\lambda_{\mathrm{S}_j}$.
The advantage of this, applied to the theorem, is that we only have to calculate  ${\mathrm{\beta}(\mathrm{S}_j)}$   and it is not necessary to calculate ${\mathrm{\beta}(\mathrm{S}_{k+j})}$.

The reader can ask him/herself that why it is not enough to compare $\lambda_{\mathrm{S}_{k+j}}$ with $\lambda _{\mathrm{S}_j}$  but we have to compare 
$\lambda_{\mathrm{S}_{k+j}}$ with $\lambda _{\mathrm{S}_j\mathrm{RL}^q}$. The reason is that it is possible to have $\mathrm{S}_{k+j}=\mathrm{S}_j\mathrm{RL}^n\mathrm{Q}, \ n=0 \dots , q-1.$  Consequently there exists $n$ such that

$$\begin{array}{r|l}
 {\mathrm{P}}= \dots \dots \dots   & {{\mathrm{S}_j(\mathrm{R} \mathrm{L}^q)}} (\mathrm{RL}^q) ^{n_{j+1}-1} \ldots )\\
  &\   \\
 \sigma ^n (\mathrm{P}) =\ldots & 
 {{\mathrm{S}_j(\mathrm{R} \mathrm{L}^n)\mathrm{Q}}}  \ldots )
 \\ 
\end{array}
$$
so we have to compare $\lambda_{\mathrm{S}_{k+j}}$ with $\lambda _{\mathrm{S}_j\mathrm{RL}^q}$.
  
   \smallskip
   
   Notice that Theorem \ref{t4} allows an arbitrary number of repeated sequences which can, and usually will, be different. In fact, Theorem \ref{t4} gives a partial solution of our goal as it is shown in the following corollary.
   
\begin{cor}{\label{c1}} Let    
$\mathrm{P}=\mathrm{R} \mathrm{L}^{q}\, \mathrm{S}_1 \cdots\mathrm{S}_k (\mathrm{R} \mathrm{L}^{q})^{n_{k+1}}\mathrm{S}_{k+1} \cdots $ $ (\mathrm{R} \mathrm{L}^{q})^{n_r}\mathrm{S}_r \mathrm{C}$ be  such that $\mathrm{S}_k\neq \mathrm{S}_1$ and    $\lambda_{(\mathrm{S}_1\mathrm{RL}^q)}<\lambda_{\mathrm{S}_k}$ for each $k\geq 2$, then $\mathrm{P}${ are MSS-sequences.} 

\end{cor} 

In broad terms, Theorem \ref{t4} 
controls the repeated sequences ending in  $(\mathrm{R} \mathrm{L}^{q})^{n_j}$. 
But, as we had already noticed in the introduction, we need also to control  $n_k$ of $(\mathrm{R} \mathrm{L}^{q})^{n_k}$. 
Next theorem will do this for the repeated sequences ending in groups $\mathrm{ S}_j$.

 \begin{thm}{\label{t5}} Let    
$\mathrm{P}=\mathrm{R} \mathrm{L}^{q}\, \mathrm{S}_1(\mathrm{R} \mathrm{L}^{q})^{n_2}   \cdots (\mathrm{R} \mathrm{L}^{q})^{n_r}\mathrm{S}_r \mathrm{C}
$ with $k,$  $1\leq k<r$,   be  {admissible sequences} such that
\begin{equation}
\mathrm{S}_{k+1}(\mathrm{R} \mathrm{L}^{q})^{n_{k+2}}\ldots (\mathrm{R} \mathrm{L}^{q})^{n_{k+j}}\mathrm{S}_{k+j}
=\mathrm{S}_1(\mathrm{R} \mathrm{L}^{q})^{n_2}\ldots (\mathrm{R} \mathrm{L}^{q})^{n_j}\mathrm{S}_j
\end{equation}
with   \ 
$(\mathrm{RL}^q)^{n_{j+1}}\neq (\mathrm{RL}^q) ^{n_{k+j+1}}
$ and $\beta((\mathrm{R} \mathrm{L}^{q})^{n_{j+1}})$ even (odd). If 
 ${n_{k+j+1}>n_{j+1}}$ and  ${n_{j+1}}$  odd (even) or \,  ${n_{k+j+1}<n_{j+1}}$ and  ${n_{k+j+1}}$ even (odd)  then  $\lambda_{\sigma^n(\mathrm{P})}<\lambda_{\mathrm{P}}$ where $\sigma^n(\mathrm{P})= \mathrm{RL}^q \mathrm{S}_{k+1}\ldots \mathrm{S}_r\mathrm{C}$.

\end{thm}

\begin{proof}
a) $\beta((\mathrm{R} \mathrm{L}^{q})^{n_{j+1}})$ even  

Let $\mathrm{\sigma}^n (P)=\mathrm{R} \mathrm{L}^{q}\mathrm{S_{k+1}}\ldots$, with $2 \, \leq k+1 \, \leq r$.  If  ${n_{j+1}}$ is odd it results 
  $$
  \begin{array}{l}
  \lambda _{(\mathrm{R} \mathrm{L}^{q})^{n_{j+1}}} =
  ( 1_{q+1},   -1_{q+1} \ldots 1_{q+1})
  \end{array}
  $$
 In a  similar way we get  $\lambda_{(\mathrm{R} \mathrm{L}^{q})^{n_{j+1}}\mathrm{S}_{j+1}}=( 1_{q+1},   -1_{q+1} \ldots 1_{q+1}, -1_{h}, \ldots) $  with  $h< q+1$ 
 where the $-1_h$ comes from the $S_{j+1}$.
As ${n_{k+j+1}>n_{j+1}}$  and $\beta(\mathrm{R} \mathrm{L}^{q})^{n_{j+1}}$ is even   it follows that

 $$\begin{array}{r|l}
  \lambda_P= (1_{q+1}, - \lambda_{\mathrm{S}_1(\mathrm{R} \mathrm{L}^{q})^{n_2}\cdots\mathrm{S}_j}, \underbrace{1_{q+1} -1_{q+1}\cdots 1_{q+1}}_{\substack{ n_{j+1}}} & \quad -1_h, \ldots\ldots\ldots) \\
         \mbox{} {\scriptsize \textrm{(by hypothesis of the theorem)}} \hspace{0.5 cm} {\Huge \mid \, \mid}  \hspace{3cm} \mbox{}     &\ {\mathbf{\vee}}  \\
   \sigma^n(\mathrm{P})=(1_{q+1}, - \lambda_{\mathrm{S_{k+1}}(\mathrm{RL}^q)^{n_{k+2}}\cdots\mathrm{S_{k+j}}}, \overbrace{1_{q+1} -1_{q+1}\cdots 1_{q+1}}^{\substack{ n_{j+1}}} & \overbrace{-1_{q+1}, \ldots}^{\substack{(n_{k+j+1}-n_{j+1})}}\ldots\ldots)
      \\ 
   \end{array} 
$$

$$
\Longrightarrow \lambda_{\sigma^n(\mathrm{P})}< \lambda_P
$$

\medskip

 If  ${n_{k+j+1}<n_{j+1}}$ and  ${n_{k+j+1}}$ is even, the proof is done in a similar way.
 
 b) $\beta((\mathrm{R} \mathrm{L}^{q})^{n_{j+1}})$ odd.  
 The proof is straightforwardly adapted from the one given in case (a).
 
 \end{proof}

Theorem \ref{t4}   controls the only shifts we have to pay attention to (see Remark \ref{r3}) in order to Theorem \ref{t1} be satisfied  when the common chunks of $\mathrm{P}$ are followed by different $\mathrm{S}_k$ groups, whereas Theorem \ref{t5} controls the only shifts we have to pay attention to (see Remark \ref{r3}) in order to Theorem \ref{t1} be satisfied  when the common chunks of $\mathrm{P}$ are followed by different $\mathrm{RL}^q$ groups.
As {admissible} sequences $\mathrm{P}=\mathrm{R} \mathrm{L}^{q} \mathrm{S}_1$ $(\mathrm{R} \mathrm{L}^{q})^{n_2}\mathrm{S}_2 (\mathrm{R} \mathrm{L}^{q})^{n_3}\mathrm{S}_3 \cdots\mathrm{S}_k (\mathrm{R} \mathrm{L}^{q})^{n_{k+1}} \mathrm{S}_{k+1} \cdots (\mathrm{R} \mathrm{L}^{q})^{n_r}\mathrm{S}_r \mathrm{C} $ result from linking $\mathrm{RL}^q$ and $\mathrm{S}_k$ groups, the only shifts that should be done in order to check whether $\sigma ^n (\mathrm{P})$ satisfies Theorem \ref{t1} are those given by Theorems \ref{t4} and \ref{t5}.
If, in addition, when  making all possible shifts Theorem \ref{t1} is verified, according to Theorems \ref{t4} and \ref{t5} it results that $\mathrm{P}$ is { MSS-sequence} and so we have the following theorem.

\begin{thm}{\label{5}}  
$\mathrm{P}=\mathrm{R} \mathrm{L}^{q}\, \mathrm{S_1}$ $(\mathrm{R} \mathrm{L}^{q})^{n_2}\mathrm{S}_2  \cdots (\mathrm{R} \mathrm{L}^{q})^{n_r}\mathrm{S}_r \mathrm{C}
$   {are MSS-sequences} if and only if for each $k$,  $1\leq k<r$, such that $\sigma^n(\mathrm{P})= \mathrm{RL}^q \mathrm{S}_{k+1}\ldots \mathrm{S}_r\mathrm{C},$  either

a) $$
\mathrm{S}_{k+1}(\mathrm{R} \mathrm{L}^{q})^{n_{k+2}}\mathrm{S}_{k+2}\ldots (\mathrm{R} \mathrm{L}^{q})^{n_{k+j}}
=\mathrm{S}_1(\mathrm{R} \mathrm{L}^{q})^{n_2}\mathrm{S}_2\ldots (\mathrm{R} \mathrm{L}^{q})^{n_j},  \ 
\mathrm{S}_j\neq \mathrm{S}_{j+k}
$$
and
$(-1)^{\beta (\mathrm{S}_j)} \lambda _{ \sigma ^n (\mathrm{P})}< (-1)^{\beta (\mathrm{S}_j)} \lambda _\mathrm{P}$ 
or 
\medskip

b)
\begin{multline*}
\mathrm{S}_{k+1}(\mathrm{R} \mathrm{L}^{q})^{n_{k+2}}\mathrm{S}_{k+2}\ldots (\mathrm{R} \mathrm{L}^{q})^{n_{k+j}}\mathrm{S}_{k+j} \\=\mathrm{S}_1(\mathrm{R} \mathrm{L}^{q})^{n_2}\mathrm{S}_2\ldots (\mathrm{R} \mathrm{L}^{q})^{n_j}\mathrm{S}_j,  \ 
(\mathrm{RL}^q)^{n_{j+1}}\neq (\mathrm{RL}^q) ^{n_{k+j+1}}
\end{multline*}
and 

$\beta((\mathrm{R} \mathrm{L}^{q})^{n_{j+1}})$ even (odd) with either 
\ ${n_{k+j+1}>n_{j+1}}$ ,   ${n_{j+1}}$  odd (even) or \,  ${n_{k+j+1}<n_{j+1}}$,  ${n_{k+j+1}}$ even (odd) 
\end{thm}

\section{Construction of blocks  $\mathbf{S(m_i,q-1)}$}

Given a sequence $\mathrm{P}$ of period $p$, 
Theorem  \ref{t4} (see part (a) of its proof) indicates that $\mathrm{S_i}  , \ i>1$ is determined by $\mathrm{S_1 RL}^q$ in order to have an MSS-sequence. 
So, we have to calculate   $\mathrm{S_1}$ as $\mathrm{S}(m,q-1)$ and, after that,  $\mathrm{S}_i ,  i>1.$

\smallskip

a) {  Construction of  $\mathrm{S}(m_i,q-1)$ with  $i=1$, $m_i=m_1$.}

\smallskip

We begin providing an algorithm to construct $\mathrm{S}(m,q-1).$  
Notice that its construction is equivalent to solving the problem of filling a row of $m$ boxes each of them with one letter 
$\mathrm{R}$  or  $\mathrm{L}$ 
in such a way that the row always starts with  $\mathrm{R}$  and has at most  $q-1$ consecutive $\mathrm{L}$s.
To get it we write  ${m=j q+r}, \ 0 \leq r<q.$ 
That is, we group the boxes in blocks  $\mathrm{B}_i, \ i=1, \ldots, j,$  of $q$ consecutive boxes each,
where each block must have at least one $\mathrm{R}$ 
but, perhaps, for the last block
 $\mathrm{F}$, which contains the last  $r$ boxes. So

 $$\mathrm{S}(m,q-1)=\mathrm{A}_1 \mathrm{A}_2 \ldots \mathrm{A}_m=\mathrm{B}_1 \mathrm{B}_2 \ldots \mathrm{B}_j \mathrm{F}$$

 \begin{center}
   ${\mathrm{S}(m,q-1)}=\stackrel{\mathrm{B}_1}{\text{\framebox[1.8cm]{${\mathrm{A}_{1} \ldots \mathrm{A}_{q}}$}}} \quad \stackrel{\mathrm{B}_2}{\text{\framebox[2.2cm]{${\mathrm{A}_{q+1} \ldots \mathrm{A}_{2q}}$}}}\ \ldots \ \stackrel{\mathrm{B}_j}{\text{\framebox[3cm]{${\mathrm{A}_{(j-1)q+1} \ldots \mathrm{A}_{jq}}$}}}\quad \stackrel{\mathrm{F}}{\text{\framebox[2.7cm]{${\mathrm{A}_{jq+1} \ldots \mathrm{A}_{jq+r}}$}}}$
\end{center}

We distinguish the following cases:

\begin{itemize}
\item [i)] \  $j=0$, i.e.
   $\stackrel{\mathrm{F}}{\text{\framebox[1.8cm]{${\mathrm{A}_{1} \ldots \mathrm{A}_{r}}$}}}$

 Since every sequence ${\mathrm{S}(m,q-1)}$ 
 is always preceded by a block   $\mathrm{RL}^q$, in order to avoid a sequence  $\mathrm{L}^{q+1}$ 
 the first symbol in  $\mathrm{F}$  must be $\mathrm{R}$, followed by a sequence of  $\mathrm{R}$s and  $\mathrm{L}$s with length   $r-1<q.$ That is, an $\mathrm{R}$  followed by the variations with repetition of  $\mathrm{L}$s and   $\mathrm{R}$s of length $r-1.$

\item [ii)] $j\neq 0.$ 
We denote by $f_i$ 
the position of the first
$\mathrm{R}$  in block $\mathrm{B_i}$ and by  $l_i$ the position of the last $\mathrm{R}$ in that block. Notice that  $f_1=1$ by definition of $\mathrm{S}(m,q-1)$ and that $f_i$ 
has not to be $1$ in the other blocks. 
Thus, the sequences of consecutive blocks
$\mathrm{B}_i$ and  $\mathrm{B}_{i+1}$ 
have the following structure with $i\neq 1$ (have in mind that before the first  $ \mathrm{R}$  
and after the last
$ \mathrm{R}$  
it is only possible to have 
 $\mathrm{L}$s) 

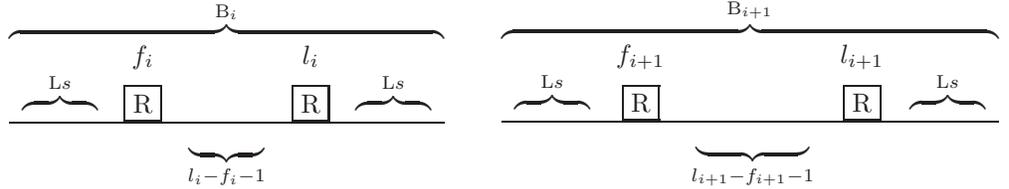
$$\overbrace{
\begin{array}{rcccl}
& f_i& &l_i & \\
 \overset{\mathrm{L}s}{\overbrace{\hspace*{1.cm}}}&{\text{\framebox[0.5cm]{$\mathrm{R}$}}}& &{\text{\framebox[0.5cm]{$\mathrm{R}$}}}& \overset{\mathrm{L}s}{\overbrace{\hspace*{1.cm}}}\\ \hline 
 & &\underset{{l_i-f_i-1}}{\underbrace{\hspace*{1.cm}}}& & \\
\end{array}}^{\mathrm{B}_i} \qquad 
\overbrace{
\begin{array}{rcccl}
   & f_{i+1}& & l_{i+1}\\
 \overset{\mathrm{L}s}{\overbrace{\hspace*{1.cm}}}&{\text{\framebox[0.5cm]{$\mathrm{R}$}}}& &{\text{\framebox[0.5cm]{$\mathrm{R}$}}}& \overset{\mathrm{L}s}{\overbrace{\hspace*{1.cm}}}\\ \hline 
 & &\underset{{l_{i+1}-f_{i+1}-1}}{\underbrace{\hspace*{1.5cm}}}& & \\
\end{array}}^{\mathrm{B}_{i+1}}
$$

with $l_i= f_i,\ldots,q$, and  $f_{i+1}=1, \ldots, l_i$.
Thus the number of consecutive $\mathrm{Ls}$ between the last   $\mathrm{R}$ of $\mathrm{B}_i$ and the first of $\mathrm{B}_{i+1} $ is, at most,  $q-1$. So the sequences in the blocks follow the pattern $\mathrm{L}^{f_i-1}\mathrm{R} \ (VR_{(2,l_i-f_i-1)}) \mathrm{R L}^{q-l_i}$, where  ${ (VR_{(2,l_i-f_i-1)}})$ denote the set of variations with repetition of $\mathrm{L}$s  and $\mathrm{R}$s with length $l_i - f_i-1$.
Notice that
$l_i-f_i-1$ can be 0; in that case the first and last  $\mathrm{R}$ in $\mathrm{B_i}$ coincide, i.e. the block $\mathrm{B}_i$ contains only one $\mathrm{R}$, so  ${VR_{(2,l_i-f_i-1)}}=\emptyset$  and $\mathrm{B}_i=\mathrm{L}^{f_i-1}\mathrm{R L}^{q-l_i}$. 

It remains to construct the block $\mathrm{F}.$ 
To do so we consider ${l_j}$,   
the position of the last $\mathrm{R}$ in the block  $\mathrm{B}_j$  previous to $\mathrm{F}.$

\begin{itemize}
\item [a)]  $l_j< r+1.$ The block $\mathrm{F}$
behaves as blocks
 $\mathrm{B}_i$ with $f_\mathrm{F}=1, \ldots, l_j$ and $l_\mathrm{F}=f_\mathrm{F}, \ldots, r.$
\item [b)]  $l_j\geq r+1$.  If block $\mathrm{F}$ 
were composed only with $r$ consecutive $\mathrm{L}$s  -the most unfavorable case-
then the number of consecutive  $\mathrm{L}$s from the last  $\mathrm{R}$ in  $\mathrm{B}_j$  would be, at most,  $q-l_j+r\leq q-(r+1)+r=q-1.$ Thus in that case
$\mathrm{F}$
will be the sequences formed by variations with repetition of
$\mathrm{R} $s and $\mathrm{L} $s with length $r$, i.e.  ${VR}_{(2,r)}.$
\end{itemize}
\end{itemize}

b) \ Construction of $\mathrm{S}_i=\mathrm{S}(m_i,q-1)$ with  $i\neq 1.$

Let us construct   $\mathrm{S}_i=\mathrm{S}(m_i,q-1)$  so that sequences $\mathrm{P} $ are { shift maximal.} 
As $\lambda _\mathrm{P} =(1_{q+1}, -\lambda _{\mathrm{S}_1}, \dots, (-1) ^{\beta (\mathrm{S}_i)} \lambda _{\mathrm{S}_i} \dots )$ and 
$\lambda _{\sigma ^n (\mathrm{P})} =(1_{q+1}, -\lambda _{\mathrm{S}_i}, \dots )$,
according to Theorem \ref{t4},
we must construct $\mathrm{S}_i$ such that $-\lambda_{\mathrm{S}_i}<-\lambda_{\mathrm{S}_1 \mathrm{RL^q}}$ in order to satisfy Theorem 1. In other words, we want $\lambda_{\mathrm{S}_1 \mathrm{RL^q}}<\lambda_{\mathrm{S}_i}.$ To get it we look for the positions in $\lambda_{\mathrm{S}_1\mathrm{RL^q}}$ where $-1$ appears and we substitute that $-1$ with $1$, such that the number of consecutive $1$s will be less or equal than  $q.$  

 Let be 
 $\lambda_{\mathrm{S}_1\mathrm{RL}^q}=(a_1,  \ldots,  a_{j}  \ldots {a_{m_1}} ,   (-1)^{\beta(\mathrm{RL}^q)}
1_{q+1})$.
 Then, with the aim of constructing $\mathrm{S}_i$ :
 
$ 1)$ Let \   $a_j$ be the first  $-1$ in $\lambda_{\mathrm{ S}_1 \mathrm{RL}^q}$  (notice that  $j>1$) which is preceded by a sequence $1_k$,  $0\leq k \leq q-1$.

$2)$ We construct $\lambda_{\mathrm{ S}_i}=( a_1, \ldots, a_{j-1}, 1, \lambda_{\mathrm{Q}}),$ where $\mathrm{Q}$  is
a sequence consisting on $\mathrm{R}$s  and $\mathrm{L}$s that has $q-1$ consecutive $\mathrm{L}$s at most whose length $l$ can vary from $0$ to a value $M$,  chosen such that the period $p$ of a sequence $\mathrm{P}$ is not beaten.
In order to avoid sequences $1_{q+1}$, after replacing $a_j$ by $1$, we have to reject those $\lambda_Q$ starting with sequences $1_n$ such that $1_k \, 1 \, 1_n = 1_h$, $h>q$.

$3)$ While $j\leq m_1+q$  we look for the next $a_j=-1$ ,  preceded by a sequence $1_k$,  $0\leq k \leq q-1$, and go back to step $2$.

Notice that if a $1_q$ is generated in the process then it will have the associate sequence $\mathrm{RL}^{q-1}$, i.e, letter $\mathrm{R}$, not only letter $\mathrm{L}$,  plays a role when sequences of $1$s are generated.

 \section{Structure of  non-primary MSS-sequences }
An important topic in dynamical systems is the  composition of sequences \cite{De}, which carries the inverse problem of knowing whether an sequence is primitive or not.

 We begin remembering  the composition law formulated by B. Derrida, A. Gervois and  Y. Pomeau.

\begin{defi} \cite{De} Let be $\mathrm{O_s=y_1\ldots y_{s-1} C}$ and $\mathrm{O_h=x_1\ldots x_{h-1} C}$, with $\mathrm{y_i, \ x_j}$ either $\mathrm{R}$ or   $\mathrm{L}$.

$\mathrm{O_h}* \mathrm{O_s}={{{\mathrm{Q} \   \mathrm{y_1} }}\ \mathrm{Q} \ \mathrm{y_2}\ldots, \mathrm{y_{s-1} \ Q }} \ \mathrm{C}$,
with $\mathrm{Q}=\mathrm{x_1\ldots x_{h-1}}$, 
if  $\mathrm{R}-$parity of  $\mathrm{O_h}$ is even 
and
$\mathrm{O_h}* \mathrm{O_s}={\mathrm{Q \ \overline{\mathrm{y_1}}} \   \mathrm{Q \, \overline{\mathrm{y_2}}} \ \ldots \mathrm{Q \,  \overline{\mathrm{y} }_{s-1} Q}}\,  \mathrm{C}$  with  ${\overline{\mathrm{y_i}}} \neq \mathrm{y_i}$, otherwise. Where $\overline{\mathrm{R}}={\mathrm{L}}$ and $\overline{\mathrm{L}}={\mathrm{R}}$
\end{defi}

\begin{Remark} 
This composition law is not restricted to the class of unimodal maps given in theorem \ref{t2}, which has been used to construct the explicit form of the MSS-sequences.
\end{Remark}
For recursive decomposition reasons, it is necessary to know the decomposition of the MSS-sequences given by 
   Theorem \ref{t2}  ($\mathrm{RL}^q \mathrm{S}(m,q-1)\mathrm{C}$).

\begin{thm}{\label{T:descomposicion1}} Given  $\mathrm{P}=\mathrm{RL}^q \mathrm{S}(m,q-1)\mathrm{C}$ a MSS sequence of length $p$,
the structure of non-primary sequences falls exclusively within the  pattern $$\mathrm{RL}^{q} \,  \mathrm{RL}^{(q-1)} \, (\mathrm{R}^2\mathrm{L}^{(q-1)})^{r} \mathrm{C},$$ {$p=(r+2)(q+1).$}
Furthermore the non-primary sequences can be written as
\begin{equation}
\mathrm{RL}^{q} \,  \mathrm{RL}^{(q-1)} \, (\mathrm{R}^2\mathrm{L}^{(q-1)})^{r} \mathrm{C} \\ \notag 
=  \mathrm{RL}^{q-1}\mathrm{C} * \mathrm{RL}^{\frac{p}{q+1}-2}\mathrm{C} \label{ec_compo}
\end{equation}
where  $q+1$ divides $p,$ \  {$q+1 \ne 1,  p.$ }
\end{thm}

\begin{proof}
Let be $\mathrm{RL}^q \mathrm{S}(m,q-1)\mathrm{C}=\mathrm{O_h}* \mathrm{O_s}$. As  $\mathrm{RL}^q \mathrm{S}(m,q-1)\mathrm{C}$ begins with   $\mathrm{RL}^q $ the only way to generate the block $\mathrm{RL}^q $ by the composition is  that either   $\mathrm{O_h}=\mathrm{RL}^q\mathrm{C}$ or   $\mathrm{O_h}=\mathrm{RL}^{q-1}\mathrm{C}$.  If  $\mathrm{O_h}= \mathrm{RL}^q\mathrm{C}$ then  the block $\mathrm{RL}^q $ would appear more than once in the non-primary sequence,  and this  would be in contradiction with the  structure of    $\mathrm{RL}^q \mathrm{S}(m,q-1) \mathrm{C}$.

 If we compose $\mathrm{O_h}=\mathrm{RL}^{q-1}\mathrm{C}$ and   $\mathrm{O_s=y_1,\ldots, y_{s-1} C}$ (where $\mathrm{y_1}=\mathrm{R}$ and $\mathrm{y_2}=\mathrm{L}$ because $\mathrm{O_s}$ is a MSS-sequence)  as the  $\mathrm{R}$-parity of $\mathrm{O_h}$ is odd, it results    
$$
\mathrm{O_h}* \mathrm{O_s}=  \mathrm{RL}^{q-1} \, {\mathrm{\overline{y_1}}} \, \mathrm{RL}^{q-1}  \, {\mathrm{\overline{y_2}}}  \, \mathrm{RL}^{q-1}   \, {\mathrm{\overline{y_{3}}}}  \ldots \, {\mathrm{\overline{y_{s-1}}}}  \mathrm{RL}^{q-1}\mathrm{C}
$$
As $\mathrm{O_h}*\mathrm{O_s}$ begins with $\mathrm{RL}^q$   it follows that  ${\mathrm{\overline{y_{1}}}} $  has necessarily  to be turned into an  $\mathrm{L}$ and, in order to avoid to get more than one  $\mathrm{RL}^q$ block, the remaining  ${\mathrm{\overline{y_{i}}}}, i=2,\ldots s-1 $ have to become  $\mathrm{R}$s, so

\begin{multline*}\mathrm{O_h}* \mathrm{O_s}=  \mathrm{RL}^{q-1} \, \underbrace{{\mathrm{\overline{y_1}}} }_{\mathrm{L}} \, \mathrm{RL}^{q-1}  \, \underbrace{{\mathrm{\overline{y_{2}}}} }_{\mathrm{R}} \, \mathrm{RL}^{q-1} \underbrace{{\mathrm{\overline{y_{3}}}} }_{\mathrm{R}}  \,  \ldots \underbrace{{\mathrm{\overline{y_{s-1}}}} }_{\mathrm{R}}  \,  \mathrm{RL}^{q-1}\mathrm{C}=
\\
= \, \mathrm{RL}^{q} \, \mathrm{RL}^{q-1}  \,  \mathrm{R}^2\mathrm{L}^{q-1}  \, \,  \mathrm{R}^2\mathrm{L}^{q-1}  \ldots  \mathrm{R}^2\mathrm{L}^{q-1}{\mathrm{C}} \, 
= \, \mathrm{RL}^{q-1}\mathrm{C} * \mathrm{RL}^{p/(q+1)-2}\mathrm{C} ,
\end{multline*}
with,   $\mathrm{O_s}= \mathrm{RL}^{p/(q+1)-2}\mathrm{C.}$

We conclude that, by construction, there only exist the described $\mathrm{O_h}$ and $\mathrm{O_s}$ whose composition generates the sequence
$\mathrm{RL}^q \mathrm{S}(r,q-1)\mathrm{C}$. 
Moreover, we get that $q+1$ must be a proper divisor of $p.$ 
\end{proof}

 \medskip

Let us focus now on the sequences with a repeated group $\mathrm{RL}^{q}$.
From the composition of the sequences
 $\mathrm{O_h}=\mathrm{R} \mathrm{L}^{q}\, \mathrm{S}_1(\mathrm{R} \mathrm{L}^{q})^{n_2}\mathrm{S}_2   \cdots (\mathrm{R} \mathrm{L}^{q})^{n_r}\mathrm{S}_r \mathrm{C}
$ and $\mathrm{O_s}=\mathrm{y_1} \, \mathrm{y_2}\ldots 
\mathrm{y_{s-1}} \,\mathrm{C}$ it follows
\begin{multline}\label{ec6} \mathrm{O_h}* \mathrm{O_s}= \mathrm{R} \mathrm{L}^{q}\, \mathrm{S}_1(\mathrm{R} \mathrm{L}^{q})^{n_2}\mathrm{S}_2   \cdots (\mathrm{R} \mathrm{L}^{q})^{n_r}\mathrm{S}_r \ {\mathrm{\overline{y_1}}}\ \mathrm{R} \mathrm{L}^{q}\, \mathrm{S}_1(\mathrm{R} \mathrm{L}^{q})^{n_2}\mathrm{S}_2   \cdots (\mathrm{R} \mathrm{L}^{q})^{n_r}\mathrm{S}_r \ {{\mathrm{\overline{y_{2}}}} } \  \cdots \\
\cdots \mathrm{R} \mathrm{L}^{q}\, \mathrm{S}_1(\mathrm{R} \mathrm{L}^{q})^{n_2}\mathrm{S}_2   \cdots (\mathrm{R} \mathrm{L}^{q})^{n_r}\mathrm{S}_r \ {\mathrm{\overline{y_{s-1}}}} \ \mathrm{R} \mathrm{L}^{q}\, \mathrm{S}_1(\mathrm{R} \mathrm{L}^{q})^{n_2}\mathrm{S}_2   \cdots (\mathrm{R} \mathrm{L}^{q})^{n_r}\mathrm{S}_r \mathrm{C}
\end{multline}
This implies that, when a  MSS-sequence has repeated the subsequences that begin with $\mathrm{R} \mathrm{L}^{q}$ and have the last character different, it is a non-primary sequence, and it factors as described in (\ref{ec6}). 
This conclusion might lead to an error since there are 
 MSS-sequences in which the repeated subsequence is not evident. A particularly interesting case is shown in the next theorem, due to its importance and later use.}

\begin{thm}{\label{T:descomposicion}}
The  MSS-sequence  $\mathrm{P}=\mathrm{R} \mathrm{L}^{q}\, \mathrm{S_1}$ $\mathrm{R} \mathrm{L}^{q}\mathrm{S}_2  \cdots \mathrm{R} \mathrm{L}^{q}\mathrm{S}_r \mathrm{C}
$  is non-primary  if and only if  \, $\mathrm{S_1}=\mathrm{S_r}\, \mathrm{L}$, \ $\mathrm{S_2}=\mathrm{S_r} \, \mathrm{R}$  \,  and  \, $\mathrm{S_i}=\mathrm{S_r}\, \mathrm{z_i}$ \, where $\mathrm{z_i}=\mathrm{R}$  or  $\mathrm{ L}$ for all   $3\leq i\leq r-1$. 
 Moreover $\mathrm{S_r}$ does not finish with $\mathrm{R}\mathrm{L}^{q-1}$.
\end{thm}

\begin{proof}
From the  $*-$composition law it follows that $\mathrm{P}$ is a non-primary sequence if and only if
  \, $\mathrm{S_i}=\mathrm{S_r}\, \mathrm{z_i}$ where $z_i=\mathrm{R}$ \ or \ $\mathrm{L}$  for   $i=1, \ldots, r-1$. So 
\begin{multline}\label{ec9} \mathrm{P}= \mathrm{R} \mathrm{L}^{q}\, \underbrace{\mathrm{S_r}\mathrm{z_1}}_{\mathrm{S}_1}\ \mathrm{R} \mathrm{L}^{q}\, \underbrace{\mathrm{S_r}\mathrm{z_2} }_{\mathrm{S}_2}\  \ \ \cdots \ \mathrm{R} \mathrm{L}^{q} \,  \underbrace{\ \mathrm{S_{r}}\mathrm{z_{r-1}} }_{\mathrm{S}_{r-1}}\ \  \mathrm{R}\mathrm{L}^{q} \underbrace{\mathrm{S}_r}_{\mathrm{S}_{r}} \mathrm{C} =  \mathrm{O_h}   *  \mathrm{y_{1}} \ {\mathrm{y_{2}}} \ \cdots {\mathrm{y_{s-1}}} \   \mathrm{C}
\end{multline}
with $\mathrm{O_h}= \mathrm{R} \mathrm{L}^{q}\, \mathrm{S_r}\ \mathrm{C}$  and $\mathrm{z_{i}}=\mathrm{y}_{i} (\mathrm{\overline{y}_{i}})$ if the $\mathrm{R}-$parity of $\mathrm{O_h}$ is even (odd).
Moreover $\mathrm{S_r}$ can not finish with $\mathrm{R} \mathrm{L}^{q-1},$ since if it happened $\mathrm{S_r}=\mathrm{S(m,q-1) \mathrm{R}\mathrm{L}^{q-1}}$ it would follow that\\

{ a) $\mathrm{R}-$parity of $\mathrm{O_h}$ is odd, and since  $\mathrm{{y}_{1}}= \mathrm{R}$ it would follow 
$$\mathrm{S_1} \, \mathrm{R} \mathrm{L}^{q}\, \mathrm{S}_2= \mathrm{S_r} \, \mathrm{L}\, \mathrm{R} \mathrm{L}^{q}\, \mathrm{S}_2= \mathrm{S(m,q-1)} \mathrm{R} \mathrm{L}^{q-1} \, \mathrm{L} \, \mathrm{R} \mathrm{L}^{q} \, \mathrm{S}_2= \mathrm{S(m,q-1)} \, (\mathrm{R} \mathrm{L}^{q})^2 \, \mathrm{S}_2.$$}

{ b) $\mathrm{R}-$parity of $\mathrm{O_h}$ is even, and since $\mathrm{{y}_{2}}= \mathrm{L}$ it would follow
$$\mathrm{S_2}\mathrm{R} \mathrm{L}^{q}\mathrm{S}_3= \mathrm{S_r} \mathrm{L} \mathrm{R} \mathrm{L}^{q}\mathrm{S}_3=\mathrm{S(m,q-1)}  \mathrm{R} \mathrm{L}^{q-1} \, \mathrm{L} \, \mathrm{R} \mathrm{L}^{q} \, \mathrm{S}_3=\mathrm{S(m,q-1)} \, (\mathrm{R} \mathrm{L}^{q})^2 \, \mathrm{S}_3.$$} 

{ In both cases groups  $(\mathrm{R} \mathrm{L}^{q})^2$ are generated, and they do  not appear in the sequence   $\mathrm{P}$.  A similar argument can be used for another 
$\mathrm{z_i}.$
}
  \end{proof}

{\begin{Remark}
Theorem \ref{T:descomposicion} prevents  $\mathrm{S_r}$ from ending in $\mathrm{R} \mathrm{L}^{q-1}.$  This fact shows the existence of composed sequences in which, apparently, there are not repeated subsequences.  It suffices taking $\mathrm{O_h}= \mathrm{R} \mathrm{L}^{q}\, \mathrm{S_r}\ \mathrm{C}$  such that $\mathrm{S_r}$ finish with $\mathrm{R} \mathrm{L}^{q-1}$ for obtaining that $\mathrm{O_h}* \mathrm{O_s}$ is a non-primary sequence with groups $(\mathrm{R} \mathrm{L}^{q})^2$  that hide the repeated subsequences, as we see in the following example. Let be
 $\mathrm{O_h}= \mathrm{R} \mathrm{L}^{5} \mathrm{R^2} \mathrm{RL}^{4} \mathrm{C}$ y $\mathrm{O_s}= \mathrm{R} \mathrm{L}^{2} \mathrm{ R} \mathrm{C}$
 \begin{multline*} \mathrm{O_h}* \mathrm{O_s}= \mathrm{R} \mathrm{L}^{5} \  \mathrm{R^2} \ \mathrm{R} \mathrm{L}^{4} \ \underbrace{{\mathrm{L}}}_{\mathrm{z_{1}}}  \ \mathrm{R} \mathrm{L}^{5} \  \mathrm{R^2} \ \mathrm{R} \mathrm{L}^{4} \ \underbrace{{\mathrm{R}}}_{\mathrm{z_{2}}}  \ \mathrm{R} \mathrm{L}^{5} \ \mathrm{R^2} \ \mathrm{R} \mathrm{L}^{4} \ \underbrace{{\mathrm{R}}}_{\mathrm{z_{3}}}  \ \mathrm{R} \mathrm{L}^{5} \ \mathrm{R^2} \ \mathrm{R} \mathrm{L}^{4} \ \underbrace{{\mathrm{L}}}_{\mathrm{z_{4}}}   \\ \mathrm{R} \mathrm{L}^{5} \  \mathrm{R^2} \ \mathrm{R} \mathrm{L}^{4} \ \mathrm{C}  
 =\mathrm{R} \mathrm{L}^{5} \  \underbrace{\mathrm{R^2}} \  (\mathrm{R} \mathrm{L}^{5})^2 \ \underbrace{ \mathrm{R^2}  \mathrm{R} \mathrm{L}^{4}  {\mathrm{R}}} \,  {\mathrm{R}} \mathrm{L}^{5}\ \ \underbrace{ \mathrm{R^2}  \mathrm{R} \mathrm{L}^{4}  {\mathrm{R} }}\,  \mathrm{R}\mathrm{L}^{5} \ \underbrace{ \mathrm{R^2}} (\mathrm{R} \mathrm{L}^{5})^2 \ \ \mathrm{R^2} \mathrm{R} \mathrm{L}^{4}  \mathrm{C}
\end{multline*}
  In general, given 
  $\mathrm{O_h}= \mathrm{R} \mathrm{L}^{q} \mathrm{H}\mathrm{R} \mathrm{L}^{q-1} \mathrm{C}$ y $\mathrm{O_s}= \mathrm{y_1} \ldots  \mathrm{y_{s-1}} \mathrm{C}$ it happens that
\begin{multline}\label{ec10} \mathrm{O_h}* \mathrm{O_s}= \\ 
\mathrm{R} \mathrm{L}^{q} \  \mathrm{H}  \mathrm{R} \mathrm{L}^{q-1} \ \mathrm{z_1}  \ \mathrm{R} \mathrm{L}^{q}\ \mathrm{H}  \mathrm{R} \mathrm{L}^{q-1} \, \mathrm{z_2} \ \mathrm{R} \mathrm{L}^{q}\cdots \cdots   \mathrm{H} \mathrm{R} \mathrm{L}^{q-1} \ \mathrm{z_{i}} \ \mathrm{R} \mathrm{L}^{q} \cdots   \mathrm{H} \mathrm{R} \mathrm{L}^{q-1} \ \mathrm{z_{j}} \ \mathrm{R} \mathrm{L}^{q} \cdots \\
  \cdots \    \mathrm{H} \mathrm{R} \mathrm{L}^{q-1} \, \mathrm{z_{s-1}} \   \mathrm{R}\mathrm{L}^{q} \mathrm{H} \mathrm{R}\mathrm{L}^{q-1} \mathrm{C}= \\
\mathrm{R} \mathrm{L}^{q}\mathrm{H}  \mathrm{R} \mathrm{L}^{q-1} \,  {\mathrm{z_1}} \, \mathrm{R} \mathrm{L}^{q} \,     \  \ \mathrm{H}  \mathrm{R} \mathrm{L}^{q-1} \,  \mathrm{z_2} \ \mathrm{R}\mathrm{L}^{q} \cdots \underbrace{\mathrm{H}  \mathrm{R} \mathrm{L}^{q-1} \,  \mathrm{R}}\   \mathrm{R} \mathrm{L}^{q} \, 
\cdots \\ \cdots   \underbrace{\mathrm{H} } (\mathrm{R} \mathrm{L}^{q})^2 \,    \cdots   \mathrm{H} \mathrm{R} \mathrm{L}^{q-1} \, \mathrm{z_{s-1}} \   \mathrm{R}\mathrm{L}^{q} \mathrm{H} \mathrm{R}\mathrm{L}^{q-1} \mathrm{C}
\end{multline}
where it has been assumed that ${\mathrm{z_{i}}} =\mathrm{R}$ and  ${\mathrm{z_{j}}} ={\mathrm{L}}$ 
(i.e. the values that are possible for a $z_k$ arbitrary) and 
 ${\mathrm{z_{1}}} =\mathrm{R} ({\mathrm{L}})$,  ${\mathrm{z_{2}}} =\mathrm{L} ({\mathrm{R}})$ if $\mathrm{R}-$parity of $\mathrm{O_h}$ is even (odd). {Note that $\mathrm{S_i}$  in (\ref{ec10}) follow the pattern $\mathrm{H}$ or $\mathrm{H}  \mathrm{R} \mathrm{L}^{q-1} \,  \mathrm{R}$, except the last one (see the above example), that is different because it has not the last letter, fact that allows us to identify the possible  non-primary sequences that contain groups  $(\mathrm{R} \mathrm{L}^{q})^2$. } 
\end{Remark} }

Finally, we shall study the particular case  $\mathrm{O_h}=\mathrm{R} \mathrm{L}^{q-1}\,  \mathrm{C}$

 \begin{thm}\label{descompo}
The MSS-sequence  $$\mathrm{P}= \mathrm{R}  \mathrm{L}^{q} \,(\mathrm{R} \mathrm{L}^{q-1}\mathrm{R})^{n_1}  (\mathrm{R} \mathrm{L}^{q})^{m_1} (\mathrm{R}\mathrm{L}^{q-1}\mathrm{R})^{n_2}   (\mathrm{R} \mathrm{L}^{q})^{m_2}   \cdots (\mathrm{R}\mathrm{L}^{q-1}\mathrm{R})^{n_r}   (\mathrm{R} \mathrm{L}^{q})^{m_r} \mathrm{R} \mathrm{L}^{q-1}\mathrm{C} 
$$ with $n_1,m_1\geq 1$ and  $n_1\geq n_i$ for all $i$, is non-primary with $\mathrm{P}= \mathrm{O_h}* \mathrm{O_s}$ and $ \mathrm{O_h}=\mathrm{R} \mathrm{L}^{q-1}\mathrm{C}$.
\end{thm}

 \begin{proof}
 We are looking for a factorization $\mathrm{P}=\mathrm{O_h}*\mathrm{O_s}$. Since $\mathrm{P}$ begins  with    $\mathrm{RL}^{q}$ and ends with  $\mathrm{R} \mathrm{L}^{q-1}\mathrm{C}$, having in mind the  $*-$law composition,  necessarily $\mathrm{O_h}=\mathrm{R} \mathrm{L}^{q-1}\mathrm{C}$ and  $\mathrm{O_s}=\mathrm{y_1}\mathrm{y_2}\cdots \mathrm{y_{s-1}}\mathrm{C}$. As  ${\mathrm{{y_{1}}}}= {\mathrm{R}}$, ${\mathrm{{y_{2}}}}= {\mathrm{L}}$ and the   $\mathrm{R-}$parity of the $\mathrm{O_h}$ is odd, it results that  ${\mathrm{\overline{y_{1}}}}$  has  to be turned into an  $\mathrm{L}$ and  ${\mathrm{\overline{y_{2}}}} =\mathrm{R}$. In order to  get more than one  $\mathrm{RL}^{q}$ block, some of the remaining  ${\mathrm{\overline{y_{i}}}}, i=3,\ldots s-1 $ have to become  $\mathrm{L}$   and  $s\geq 4$. So  

$$ \mathrm{O_h} * \mathrm{O_s} = \mathrm{RL}^{q-1} \, \underbrace{{\mathrm{\overline{y_1}}} }_{\mathrm{L}} \, \mathrm{RL}^{q-1}  \, \underbrace{{\mathrm{\overline{y_{2}}}} }_{\mathrm{R}} \, \mathrm{RL}^{q-1} {{\mathrm{\overline{y_{3}}}} }  \,  \ldots \,  \mathrm{RL}^{q-1} \, {\mathrm{\overline{y_{i}}}} \, \mathrm{RL}^{q-1} \ldots {\mathrm{\overline{y_{s-1}}}} \,  \mathrm{RL}^{q-1}\mathrm{C}
$$
 
Notice that for all  $i\geq 3$  
$${\mathrm{R}\mathrm{L}^{q-1}\mathrm{\overline{y_{i}}}}=\left\{
  \begin{array}{ll}
\mathrm{R} \mathrm{L}^{q}    & \hbox{ if $\mathrm{\overline{y_{i}}}=\mathrm{L}$} \\
    \mathrm{R} \mathrm{L}^{q-1}\mathrm{R}  & \hbox{ si $\mathrm{\overline{y_{i}}}=\mathrm{R}$}
   \end{array}
\right.$$ 

Then
$$\mathrm{P}= \mathrm{R}  \mathrm{L}^{q} \,(\mathrm{R} \mathrm{L}^{q-1}\mathrm{R})^{n_1}  (\mathrm{R} \mathrm{L}^{q})^{m_1} (\mathrm{R}\mathrm{L}^{q-1}\mathrm{R})^{n_2}   (\mathrm{R} \mathrm{L}^{q})^{m_2}   \cdots (\mathrm{R}\mathrm{L}^{q-1}\mathrm{R})^{n_r}   (\mathrm{R} \mathrm{L}^{q})^{m_r} \mathrm{R} \mathrm{L}^{q-1}\mathrm{C} 
$$Note that  $n_1\geq n_i$ for each $i$ because if there is $n_i > n_1$ then $\mathrm{O_s}=\mathrm{R}  \mathrm{L}^{n_1} \ \ldots \mathrm{R} \mathrm{L}^{n_i} \ldots \mathrm{C}$ would not be a MSS-sequence according to lemma \ref{l1}.
\end{proof}

\begin{cor}\label{fact}
The non-primary sequences P factor as  $\mathrm{P}=\mathrm{O_h}*\mathrm{O_s}$ for $\mathrm{O_s}=\mathrm{y_1}\cdots ...\mathrm{y_{s-1}}\mathrm{C}$ and $\mathrm{O_h}$ one of the three following possibilities
\begin{itemize}
\item [i)] $\mathrm{O_h}=\mathrm{RL}^{q}  \mathrm{S_1} (\mathrm{RL}^{q})^{n_2} \mathrm{S_2} \cdots (\mathrm{RL}^{q})^{n_r} \mathrm{S_{r}} \  \mathrm{C}$.
\item [ii)]  $\mathrm{O_h}=\mathrm{RL}^{q}  \mathrm{S_r} \  \mathrm{C}.$
\item [iii)]  $\mathrm{O_h}=\mathrm{RL}^{q-1}  \  \mathrm{C}$
\end{itemize}
\end{cor}

\begin{Remark}
In particular, the non-primary sequences with a $(\mathrm{RL}^q)^{n_i}$ group, $\, n_i\geq 3$, for some $i$, are either of the type given by theorem \ref{descompo} or all their  $\mathrm{S_i}$ groups are repeated except, perhaps, the last one (the last letter is missing) as it is shown in expression (\ref{ec6}). 
It is important to remark that  if the 
 MSS-sequence is non-primary, the sequence  $\mathrm{O_h}$ appears just in the 
 final part of  the sequence, and this allows us to deduce who is  $\mathrm{O_h}$  and so who is $\mathrm{O_s}.$ 
 For instance, let be
  $\mathrm{P}= \mathrm{RL}^{4}(\mathrm{RL^3R})^2 (\mathrm{RL}^{4})^3 \, \mathrm{RL^3R} \, \mathrm{RL}^{3}\mathrm{C}$. We deduce that $\mathrm{O_h}=\mathrm{RL}^{3}\mathrm{C}$ y $\mathrm{O_s}=\mathrm{R}\mathrm{L^2} \mathrm{R^3}\mathrm{L}\mathrm{C}$

\bigskip

\textbf{Recursive decomposition and factorization.} Non-primary sequences factor according to corollary \ref{fact}. 
When $\mathrm{O_h}$ corresponds to case i) in corollary \ref{fact} it happens that $\mathrm{O_h}$ is  either primary or it admits one of the factorizations shown in this section.                                                                                                                                                                                
Notice that when decomposing $\mathrm{P}$ as $\mathrm{RL}^q \mathrm{S}(m_1,q-1)\mathrm{C} *
\mathrm{{y_1 } \, {y_2}\ldots {y_{s-1}} \,C} $(case ii) in corollary \ref{fact}) we have that  $\mathrm{RL}^q \mathrm{S}(m_1,q-1)\mathrm{C}$ is either primary or non-primary of type 
 $\mathrm{RL}^{q-1}\mathrm{C}* \mathrm{RL}^{\frac{p}{q+1}-2}\mathrm{C} $ in agreement with Theorem \ref{T:descomposicion1}, where sequences $\mathrm{RL}^n\mathrm{C}$ are  primary. In the case iii) of corollary \ref{fact} it is $\mathrm{O_h}=\mathrm{RL}^{q-1}\mathrm{C}$ that is a primary sequence.

On the other side, $
\mathrm{{y_1 } \, {y_2}\ldots {y_{s-1}} \,C}$  is either primary or
it admits decomposition,
 so the decomposition can follow in a recursive way generating the sequence factorization (regarding this factorization as the set of the primary sequences in which it is decomposed).
\end{Remark}

\section{Cardinality of non-primary sequences}

\subsection{ Cardinality of $\mathrm{RL^q S}(m,q-1)\mathrm{C}$ sequences.}

According to Theorem \ref{T:descomposicion1} this kind of sequence has an unique decomposition given by 
\begin{equation} \label{ec_compo2} \mathrm{RL}^{q-1}\mathrm{C} * \mathrm{RL}^{\frac{p}{q+1}-2}\mathrm{C}
\end{equation}
where  $q+1$  is a proper  divisor of $p$, the length of $\mathrm{P}$.

Notice that, according to the decomposition,  $q+1=p$ , and  $q=0$ (which corresponds to { the divisors $p$ and $1$ of $p$ respectively}) are forbidden because then negative powers would appear in expression (\ref{ec_compo2}).
For a fixed period $p$, the uniqueness of the sequence decomposition, implies that the cardinality of the set $\mathrm{RL}^q \mathrm{S}(m,q-1)\mathrm{C}$ must come from the proper divisors of  $p$, according to (\ref{ec_compo2}).

  Let   $p=\prod_{i=1}^{\rho}\, {p_i}^{a_i}$
  be the decomposition of $p$ in prime factors. Then the set of its divisors
  is  $D_p = \{ \prod_{i=1}^{\rho}\, {p_i}^{c_i}: 0\leq  c_i\leq a_i \}$.

 So the number of this type of non-primary sequences is $\text{cardinality}(D_p)-2$, where $-2$ corresponds to remove $1$ and $p$ as divisors.
 Notice that we take  $q+1$ through $D_p$.

\medskip

\subsection{Cardinality of $
\mathrm{R} \mathrm{L}^{q} \mathrm{S}(m_1,q-1)  $}$ \cdots \mathrm{R} \mathrm{L}^{q}\mathrm{S}(m_r,q-1) \mathrm{C}
$ sequences.

\smallskip

The non-primary sequences  $\mathrm{P}$ given by Theorem  \ref{T:descomposicion} are $
\mathrm{P}= \mathrm{RL}^q\mathrm{S}(m-1,q-1)\mathrm{C} * \mathrm{y_1\cdots y_ {s-1} C =} \mathrm{O}_{d} * \mathrm{O}_{p/d}$
 where $d \in D_p$. {The presence of the block $\mathrm{RL}^q$ in the sequences $\mathrm{P}$ implies $q>0$.  Since $q>0$ it follows that $d \geq 3$  (the case $d=3$ corresponds to $O_{d}= \mathrm{RLC}).$}

i) For $q$ fixed  $Card(\mathrm{P})=  Card(\mathrm{S}(m-1,q-1))$ {because, according to theorem \text{\ref{T:descomposicion}}, the sequences $\mathrm{P}$ are just the sequences resulting from joining  $p/d$ identical subsequences, where each one of them is in the set $\mathrm{RL}^q\mathrm{S(m-1,q-1)}$ with $\mathrm{RL}^q$  fixed.

ii) For $d$ fixed, the value of $q$ in $\mathrm{RL}^q\mathrm{S}(m-1,q-1) \mathrm{C}$ can take values in the set  $q= 1, \dots , d -2$, as $d$ is the period of $\mathrm{O}_{d}$ (the case $q=d - 2$ corresponds to $\mathrm{O}_{d}=\mathrm{RL^qC}$ with  $m=1$ and  $\mathrm{S(m-1,q-1)}= \emptyset$}).

On the other hand
the period of $\mathrm{RL}^q\mathrm{S}(m-1,q-1)\mathrm{C}$ is $q+(m-1)+2=d$, so $m-1=d -(q+2)$ and, for any $q$, we have

\begin{multline*}
Card(\mathrm{P}) = \sum _{q=1} ^{d -2} Card(\mathrm{S}(m-1,q-1))
= \sum _{q=1} ^{d -2}  Card(\mathrm{S}(d-(q+2),q-1))
\end{multline*}

iii) Considering every divisor  $d \in D_p$ it follows, for a fixed period $p$, that
$$
Card(\mathrm{P})=  \sum _{d \in {\bar D}_p} \sum _{q=1} ^{d -2}  Card(S(d-(q+2),q-1)) \label{sumaiii}
$$
where ${\bar D}_p$ denotes the set of proper divisors of $p$.

\medskip

Finally we have to calculate $\mathrm{S}(m,q-1)$ with $m=d-(q+2)$.
In order to count the number of sequences 
 $\mathrm{S}(m,q-1)$
 we will classify those lists regarding the number of $\mathrm{Rs}$ that contain:
 $$
\mathrm{S}(m,q-1)= \bigcup _ {k=[ \frac{m-1}{q} ]} ^{m-1} \mathrm{S}(m,q-1,k)  
$$
where $\mathrm{S}(m,q-1,k)$ denotes the set of lists with length  $m$, consisting on symbols $\mathrm{R}$ and $\mathrm{L}$, starting with $\mathrm{R}$, containing  $q-1$ consecutive $\mathrm{L}$s at most and  $k$ $\mathrm{R}$s apart from the first one.

There is a bijection between $\mathrm{S}(m,q-1,k)$ and the set of solutions of
\begin{equation} \label{Bij}
\begin{cases}
x_1+x_2+ \cdots + x_{k+1} = m-k-1 & \\
0 \le x_i\le q-1, \ \ j= 1 \dots k+1 \\
\end{cases}
\end{equation}

As
$$
|\mathrm{S}(m,q-1)|= \sum _k |\mathrm{S}_k (m,q-1) |
$$
we have to estimate the number of solutions of (\ref{Bij}).
With that objective in mind we define
$$
A_{i_1i_2\cdots i_r}= \text{solutions of }
\begin{cases}
x_1+x_2+ \cdots + x_{k+1} = m-k-1 & \\
x_i \ge q \ \forall i \in \{ i_1, \cdots, i_r\} & \\
x_j\ge 0, \ \forall j\not \in  \{ i_1, \cdots, i_r\}
\end{cases}
$$

Then, by  inclusion/exclusion, we have that
\begin{multline*}
|\mathrm{S}_k(m,q-1)|= \\ 
|\mathrm{A}_\emptyset |- \sum _ {i=1} ^{ k+1} |\mathrm{A}_i| + \sum _{i_1<i_2} |\mathrm{A}_{i_1i_2}| - \cdots + (-1)^{k} |\mathrm{A}_{12\cdots k+1}|
\end{multline*}

It is known  \cite{commet} that the number of solutions of
$$
\begin{cases}
z_1+x_2+ \cdots + z_{\alpha} =  n \\
z_i \ge p_i, i=1 \dots \alpha
\end{cases}= 
\begin{pmatrix}
n+\alpha-1-\sum_{i=1}^\alpha p_i \\ \alpha -1
\end{pmatrix}
$$
so
\begin{multline*}
|\mathrm{A}_{i_1i_2\cdots i_r}|= \\
\begin{pmatrix}
(m-1-k)+(k+1)-1-r_q \\ k
\end{pmatrix}
\begin{pmatrix}
m-1-r_q \\ k
\end{pmatrix}
\end{multline*}

Thus, since in every sum 
 $\displaystyle{\sum _{i_1<i_2< \cdots <i_r}}$ the number of terms is 
$
\begin{pmatrix}
k+1 \\ r
\end{pmatrix},
$
we get finally that 
\begin{multline*}
Card (\mathrm{S}(m,q-1))= \\
 \sum_{k=[(m-1)/q]}^m \,  \sum_{r=0}^{\min \{k+1, \frac{m-k}{q}\}}
\binom{k+1}{r} \, \binom{m-1-rq}{k}  \, (-1)^r \\
\end{multline*}
 
\bigskip

\section{Conclusion  and its relation with open-problems in Dynamical Systems}

The combinatorial descriptions of one-dimensional discrete systems $x_{n+1}=f(x_n)$  ruled by unimodal {round-top concave} functions (see (\ref{eqn1})) has a pending problem since Beyer, Mauldin and Stein gave a theorem to decide whether {an admissible sequence was an MSS-sequence or not}  (see theorem \ref{t1}) . Namely, which is the explicit expression of MSS-sequences of these dynamical systems. We have solved this problem by proving that the structure of MSS-sequences is 
 \begin{equation} \label{eqn:P-conclusiones}
 {\small
 \mathrm{P}=\mathrm{R} \mathrm{L}^{q}\, \mathrm{S}_1(\mathrm{R} \mathrm{L}^{q})^{n_2}  \mathrm{S}_{2} \cdots  (\mathrm{R} \mathrm{L}^{q})^{n_r}\mathrm{S}_r \mathrm{C}
 }
\end{equation}
That is, sequences result from  linking alternatively  $\mathrm{RL}^q$ and $\mathrm{S}_i$ blocks, where $\mathrm{S}_i$ are sequences of $\mathrm{R}$s and $\mathrm{L}$s  such that the longest consecutive sequence of $\mathrm{L}$s has $q-1$ symbols. 
Theorem \ref{5} states how  $\mathrm{S}_i$ and $(\mathrm{RL}^q)^{n_k}$ must be linked in order to (\ref{eqn:P-conclusiones}) be an MSS-sequence. Perhaps, the most striking fact we have found is that the condition for being a MSS-sequence is ruled by $S_1$,
 the subsequence contained between the first two   
$\mathrm{RL}^q$
blocks,   since the rest of $\mathrm{S}_i,\; i>1$  are built from $\mathrm{S}_1\mathrm{RL}^q$ (section IV). As the $\mathrm{RL}^q$ blocks are trivial, the building of the MSS-sequence is determined by $\mathrm{S_i}$ and therefore by $\mathrm{S_1}$. An inheritance process is manifested, that must be studied.
On the other hand, the fact that $S_i$ is derived from $\mathrm{S}_1$ implies that if $\mathrm{S}_1$ are constructed then we are able to construct $\mathrm{S}_i$, and so the sequence $\mathrm{P}.$ The construction of $S_1$ is given in Section IV (a) and $\mathrm{S}_i$ in IV (b).
The importance of $\mathrm{S}_1$ is not only shown by the fact that $\mathrm{S}_i$ groups are constructed from $\mathrm{S}_1$, but also because of the role played by $\mathrm{S}_1$ in the composition/decomposition of non-primary sequences (theorem \ref{T:descomposicion}).

The fact of having an explicit form of MSS-sequences has led us to characterize which are non-primary and how they are factorized as composition of primary sequences (Section V). In this factorization process $\mathrm{S}_1$ plays an important role: its cardinality, along with the factorization theorems, has led us to calculate the cardinality of different non-primary sequences.
Notice that factorization theorems (Section V) impose very restrictive conditions to be non-primary sequences, i.e. for sequences of fixed length primary sequences will be much more abundant than non-primary, as it had always been noticed \cite{De}.

Our results can supply new  approaches and tools for open-problems in dynamical systems. As we know the explicit expression of all MSS-sequences, we can calculate which orbits are near to each other in the phase space, and find out how they cluster. This kind of clusters plays an important role in quantum chaos \cite{boris}.

An emerging problem, related to clustering of orbits, is what we might call the inverse problem: we do not want to visit certain zones of an attractor. The MSS-sequences, lacking the sequence coding that zone, will not visit this zone \cite{Yu}.

Other uses are possible.  As MSS-sequences are known according to our theorems, we can use them to calculate where they are located in parameter space, by using the algorithm given by Myrberg \cite{ref_1}. Compare  how easy it is locating the sequences using Myrberg algorithm with the extraordinarily big computational cost required when sequences are looked for using brute force.  These big computational costs limit numerical experiments to sequences of low period \cite{ref_bajo}, so it is very useful having a simple method for locating sequences in the parameter space.

In section V we identify the non-primary sequences, and so the primary sequences. Primary cycles are important to get a more accurate approximation of strange attractors, by shadowing long orbits when cycle expansions are used \cite{artu}. 

Finally, we have calculated the cardinality of different sets of  non-primary sequences. In this knowledge lies the foundation to calculate topological entropy \cite{Ban}.

\vspace{1cm}

\textbf{Acknowledgements.-}
The second and third authors are partially supported by Ministerio de Economía, Industria y Competitividad under the grant TRA2016-76914-C3-3-P Thanks are given to Daniel Rodríguez-Pérez and Pablo Fernández-Gallardo who improved the final version of this paper. 


\vspace{1cm}


\begin{thebibliography}{99}


\bibitem{collet} (MR2541754)[10.1002/zamm.19810611022]
 \newblock P. Collet and  J.P. Eckmann, 
 \newblock \emph{Iterated Maps on the Interval as Dynamical Systems,}
 \newblock Birkhauser, Boston, 1980.

\bibitem{cv} (MR2740091)[10.1016/j.physd.2010.07.010]
 \newblock E. Siminos  and  P. Cvitanović ,
 \newblock Continuous symmetry reduction and return maps for high-dimensional flows,
\newblock \emph{Physica D},
\textbf{240} (2011), 187--198.

\bibitem{ref_1} (MR2740091) [10.5186/aasfm.1964.000]
 \newblock P. J. Myrberg,
 \newblock {Iteration der reellen polynome zweiten grades III,}
 \newblock \emph{Acad. Sci. Fenn.}, \textbf{336}, 3 (1963). 

\bibitem{ref_2} (MR2740091)[10.1016/0097-3165(73)90033-2]
 \newblock N. Metropolis, M.L. Stein  and  P.R. Stein,
 \newblock On Finite Limit Sets for Transformations on the Unit Interval,
\newblock \emph{J. Combin. Theory Ser. A,} \textbf{15}(1) (1973), 25--44.

\bibitem{beyer} (MR2740091)[10.1016/0022-247X(86)90001-6]
 \newblock W.A. Beyer, R.D.  Mauldin R D and P.R. Stein,
 \newblock Shift-maximal sequences in function iteration: Existence, uniqueness and multiplicity,
 \newblock \emph{J. Math. Anal. Appl.,} \textbf{115} (1986), 305--362.

\bibitem{wang} (MR0899901)[10.1016/0097-3165(87)90075-6]
 \newblock Li Wang and  N. D. Kazarinoff, 
 \newblock On the Universal Sequence Generated by a Class of Unimodal Functions,
 \newblock \textit{J. Combin. Theory Ser. A,} \textbf{46} (1987), 39--49.

\bibitem{ref_3} (MR0899901) [10.1007/BF01020332]
 \newblock M.J. Feigenbaum,
 \newblock Quantitative universality for a class of nonlinear transformations,
 \newblock \emph{J. Stat. Phys.,} \textbf{19} (1978), 25--52. 

\bibitem{De} (MR519698)
 \newblock B. Derrida, A.  Gervois  and Y. Pomeau,
 \newblock Iteration of endomorphisms on the real axes and representation of numbers,
 \newblock \emph{Ann. Inst. H. Poincaré Phys. Théor.,} \textbf{29} (1978), 305--356.

\bibitem{Hao2} [10.1007/s00026-000-8003-3]
 \newblock H. Bailin,
 \newblock Number of Periodic Orbits in Continuous Maps of the Interval-Complete Solution of the Counting Problem,
 \newblock \emph {Ann. Comb.,} \textbf{4} 3 (2000), 339--346.

\bibitem{cv2} (MR0890474) [10.1103/PhysRevLett.58.2387]
 \newblock D. Auerbach,  P. Cvitanović, J.P. Eckmann, G. Gunaratne and I. Procaccia,
 \newblock Exploring chaotic motion through periodic orbits,
 \newblock \emph{Phys. Rev. Lett.} \textbf{58} 23 (1988), 2387--2389.

\bibitem{commet} ( MR0460128) [10.1007/978-94-010-2196-8]
 \newblock L. Comtet,
 \newblock \emph{Advanced Combinatorics: The Art of Finite and Infinite Expansions,}
D. Reidel Pub. Co., Dordrecht, 1970.

\bibitem{boris} (MR3001767) [10.1088/0951-7715/26/1/177]
 \newblock B. Gutkin and  V. Al Osipov,
 \newblock Clustering of periodic orbits in chaotic systems,
 \newblock \emph{Nonlinearity,} \textbf{26} 1 (2013),  177--200.

\bibitem{Yu} (MR2630098) [10.1088/0951-7715/23/5/010]
 \newblock Y. Ilyashenko  and  A. Negut,
 \newblock Invisible parts of attractors,
 \newblock \emph{Nonlinearity,} \textbf{23} 5 (2010), 1199--1219.

\bibitem{ref_bajo} (MR3402871) [10.1142/S0218127415501394]
 \newblock Z. Galias,
 \newblock Rigorous Numerical Study of Low-Period Windows for the Quadratic Map,
 \newblock \emph{Int. J. Bifurcation Chaos,}  \textbf{25} 10 (2015), 1550139.

\bibitem{artu} (MR1054579) [10.1088/0951-7715/3/2/005]
 \newblock R. Artuso, E. Aurell and P. Cvitanović,
 \newblock Recycling of Strange Sets: I. Cycle Expansions,
\emph{Nonlinearity,} \textbf{3} (1990), 325--359.

\bibitem{Ban} [10.1103/PhysRevLett.88.174102]
 \newblock C. Bandt and  B. Pompe,
 \newblock Permutation Entropy: A Natural Complexity Measure for Time Series,
 \newblock \emph{Phys. Rev. Lett.,} \textbf{88} 17 (2002), 174102.

\end{thebibliography}
\end{document}